\documentclass[10pt]{amsart}
\usepackage[english]{babel}
\usepackage[pdfborder={0 0 0}]{hyperref}
\usepackage{amsmath, amssymb, amsthm}
\usepackage{stmaryrd}
\usepackage{mathtools}
\usepackage{enumitem}
\usepackage{tikz}
\usepackage{tikz-cd}
\usepackage{colonequals}
\usepackage{fullpage}
\usepackage{bbm}

\theoremstyle{plain}                            
\newtheorem{theorem}{Theorem}[section]
\newtheorem{corollary}[theorem]{Corollary}
\newtheorem{lemma}[theorem]{Lemma}
\newtheorem{proposition}[theorem]{Proposition}
\newtheorem*{theorem*}{Theorem}

\theoremstyle{definition}

\newtheorem{remark}[theorem]{Remark}

\title{Beyond the Giampietro--Darmon conjecture}

\date{\today}

\author{Michael A. Daas}
\address{Department of Mathematics, Universit\'e du Luxembourg, L-4364 Esch-sur-Alzette, Luxembourg}
\email{michael.daas@uni.lu}

\begin{document}

\begin{abstract}
In \cite{Giam}, Giampietro and Darmon conjectured a formula for the norm of various algebraic numbers, obtained as infinite products of $p$-adic cross-ratios of CM points. These quantities arose from the $p$-adic uniformisation of Shimura curves and displayed strong parallels with the Gross--Zagier factorisation for the norms of the differences between two singular moduli from \cite{GZ}. The conjectured formula was conditional on the genus of the Shimura curve being zero, and in \cite{daas1}, this formula was proved in most cases. In this work, we extend the validity of the factorisation formula beyond what was conjectured by Giampietro and Darmon to many more cases, by relating this to the genus of an Atkin--Lehner quotient of the Shimura curve being zero instead. To this end, we solve a $p$-inverted version of a counting problem that was previously considered in work of Howard and Yang \cite{HY}.
\end{abstract}

\maketitle

\setcounter{tocdepth}{1}
\tableofcontents

\section{Introduction}\label{intro}

Let $N > 1$ be a squarefree integer supported at an even number of primes and let $\mathfrak{B}$ be the indefinite rational quaternion algebra ramified precisely at the primes dividing $N$. It has a maximal order $\mathfrak{R}$ that is unique up to conjugation. If we choose a splitting $\mathfrak{B} \xhookrightarrow{} M_2( \mathbb{R} )$, the subgroup $\mathfrak{R}_1 \subset \mathfrak{R}^{\times}$ consisting of all elements of unit norm acts on the complex upper half plane $\mathcal{H}$ through M\"obius transformations. Then
\[
X_N( \mathbb{C} ) = \mathfrak{R}_1 \setminus \mathcal{H}
\]
is compact and called a Shimura curve of level $N$. It is equipped with an action of the Atkin--Lehner group, consisting of commuting involutions $\text{w}_d$ for all positive $d \mid N$. One can show that $X_N$ is in fact an algebraic curve and has a model defined over $\mathbb{Q}$. As there are formulas available for its genus, see for example \cite{ogg83} or Proposition 2.2 in \cite{oana}, one may show that the curve $X_N$ is of genus 0 if and only if $N \in \{ 6, 10, 22 \}$.

The curve $X_N$ has bad semistable reduction at all prime divisors of $N$. Choose a prime divisor $p$ of $N$ and write $N = pM$ for some positive integer $M$. A theorem of \v{C}erednik and Drinfeld describes an explicit $p$-adic uniformisation of $X_N$, which we recall now. Let $B$ denote the definite rational quaternion algebra ramified precisely at the primes dividing $M$ and infinity. Fixing a splitting $B \xhookrightarrow{} M_2( \mathbb{Q}_p )$, we obtain an action of $B^{\times}$ on the $p$-adic upper half plane $\mathcal{H}_p = \mathbb{P}^1(\mathbb{C}_p) - \mathbb{P}^1(\mathbb{Q}_p)$ by M\"obius transformations. Let $\mathcal{R} \subset B$ be a maximal $\mathbb{Z}[1/p]$-order, which is, as we will recall later, unique up to conjugation. An element in $\mathcal{R}$ is a unit if and only if its norm is a power of $p$, so if we define $\overline{ \mathcal{R} }^{\times} = \mathcal{R}^{\times} / p^{\mathbb{Z}}$, the subgroup $\Gamma \subset \mathcal{R}^{\times}$, consisting of all elements of unit norm, may be viewed as an index 2 subgroup of $\overline{ \mathcal{R} }^{\times}$. The quotient 
\[
\Gamma \setminus \mathcal{H}_p
\]
is again compact, and, as \v{C}erednik and Drinfeld proved in \cite{cerednik, drinfeld}, over $\mathbb{C}_p$, it is isomorphic to $X_N$, with the isomorphism itself being defined over $\mathbb{Q}_{p^2}$, the unique quadratic unramified extension of $\mathbb{Q}_p$. Through this isomorphism, the Galois action of elements that do not fix $\mathbb{Q}_{p^2}$ on $\Gamma \setminus \mathcal{H}_p$ is twisted by the action of the Atkin--Lehner operator $\text{w}_p$ on the Shimura curve $X_N$. For a detailed exposition of this result, we refer the reader to \cite{boucar}. For the statements of some key facts, see Section 2.3 of \cite{oana}.

These curves have a natural supply of CM points, which we now describe. Let $K$ be an imaginary quadratic field with maximal order $\mathcal{O}$ of discriminant $\Delta < 0$. Given an embedding $\alpha : \mathcal{O} \xhookrightarrow{} \mathcal{R}$, we obtain an action of $K^{\times}$ on $\mathcal{H}_p$. Such embeddings $\alpha$ exist if and only if all primes dividing the discriminant $M$ of $B$ are non-split in $\mathcal{O}$, and if in addition $p$ is non-split in $\mathcal{O}$ as well, the image of $K^{\times}$ under each such embedding has two Galois conjugate fixed points that are contained in $\mathcal{H}_p$. A CM point of discriminant $\Delta$ on $\Gamma \setminus \mathcal{H}_p$ is then defined as the image of such a fixed point in this quotient.

The function field of $\Gamma \setminus \mathcal{H}_p$ is generated by so-called $p$-adic $\Theta$-functions, see \cite{GvdP}. Given two base points $w_1, w_2 \in \mathcal{H}_p$, one defines the function $\Theta(w_1, w_2; - ) : \mathcal{H}_p \to \mathbb{C}_p$ through the formula
\[
\Theta(w_1,w_2; z) \colonequals \prod_{\gamma \in \Gamma} \frac{z - \gamma w_1}{z - \gamma w_2}.
\]
There exists a group homomorphism $c_{w_1, w_2} : \Gamma \to \mathbb{C}_p^{\times}$, called the \emph{factor of automorphy}, such that
\[
\Theta( w_1, w_2; \gamma z ) = c_{w_1, w_2}( \gamma ) \Theta( w_1, w_2; z )
\]
for all $z \in \mathcal{H}_p$ and $\gamma \in \Gamma$. The homomorphisms $c_{w_1,w_2}$ factor through the quotient $\overline{ \Gamma } \colonequals \Gamma^{\text{ab}} / \Gamma^{\text{ab}}_{\text{tors}}$ of $\Gamma$, which is abstractly isomorphic to $\mathbb{Z}^g$, where $g = g(X_N)$ denotes the genus of the Shimura curve $X_N$. In particular, if $g = 0$, this automorphy factor is trivial and may be ignored. In this case, $\Theta(w_1, w_2; z )$ descends to a function on the quotient $\Gamma \setminus \mathcal{H}_p$.

In \cite{Giam}, Giampietro and Darmon investigated the following quantity. Fix two imaginary quadratic fields $K_1$ and $K_2$ with rings of integers $\mathcal{O}_1$ and $\mathcal{O}_2$ respectively. Let $D_1 < 0$ and $D_2 < 0$ denote their discriminants and for $i \in \{ 1,2 \}$, set $w_i = \# \mathcal{O}_i^{\times}$. We assume that $\text{gcd}(D_1, D_2) = 1$ and we suppose that $p$ is inert in both $K_1$ and $K_2$, as well as all primes that ramify in $B$. Then there exist embeddings $\alpha_1 : \mathcal{O}_1 \xhookrightarrow{} \mathcal{R}$ and $\alpha_2 : \mathcal{O}_2 \xhookrightarrow{} \mathcal{R}$. For $i \in \{ 1,2 \}$, let $\tau_i, \tau_i' \in \mathcal{H}_p$ be the two CM points associated with $\alpha_i$. Now consider the quantity
\begin{equation} \label{maineq}
\frac{ \Theta( \tau_2, \tau_2'; \tau_1 ) }{ \Theta( \tau_2, \tau_2'; \tau_1') } = \prod_{ \gamma \in \Gamma } \frac{( \tau_1 - \gamma \tau_2 )( \tau_1' - \gamma \tau_2' )}{ ( \tau_1' - \gamma \tau_2 )( \tau_1 - \gamma \tau_2') }.
\end{equation}
If $N \in \{ 6, 10, 22 \}$, as now there is no factor of automorphy, the $\Theta$-functions in the expression above descend to functions on $\Gamma \setminus \mathcal{H}_p$. As in \cite{Giam}, one may then relate this infinite product to the cross-ratio of the CM values of some modular function $j_N : X_N \xrightarrow{ \sim } \mathbb{P}^1$, which are known to be algebraic by Shimura reciprocity. The authors conjectured an explicit prime factorisation formula for the norm of this algebraic number, which strongly mimics the prime factorisation of the differences between singular moduli proved by Gross and Zagier in \cite{GZ}; see Theorem \ref{mainthm} below. This archimedean version of the conjecture was first proved in \cite{daas1}, supported by the work of Phillips \cite{anphil}. Subsequently, a different proof was given by Crabit Nicolau \cite{crabit}, that displays close parallels with the analytic proof in \cite{GZ}. 

The work \cite{daas1} also contains a purely $p$-adic computation of the norm of the quantity in Equation \ref{maineq}, using a method that can be divided into three steps. First, one observes that, as a consequence of our symmetrisation, both the numerator and the denominator of the right hand side of Equation \ref{maineq} are contained in the real quadratic field $F = \mathbb{Q}( \sqrt{ D } )$, where $D = D_1D_2 > 0$. In fact, the fraction can be rewritten as $q_F( \gamma ) / q_F'( \gamma )$, where $q_F, q_F' : B \to F$ are two Galois conjugate $F$-quadratic forms that we will explore in detail in Section \ref{qFsec}. One may then rewrite Equation \ref{maineq} as
\begin{equation}\label{rewrite}
\frac{ \Theta( \tau_2, \tau_2'; \tau_1 ) }{ \Theta( \tau_2, \tau_2'; \tau_1') } = \prod_{ \nu \in F } \left( \frac{ \nu }{ \nu' } \right)^{\# \{ \gamma \in \Gamma \mid q_F( \gamma ) = \nu \} }.
\end{equation}
The first step is to relate the exponents in the above expression to the Fourier coefficients of a parallel weight $(1,1)$ Hilbert Eisenstein series $E_{1, \chi}$, which can be expressed in terms of divisor sums over $F$-ideals.

The second step is to prove an $R = T$-theorem, which we will not concern ourselves with here. This facilitates the third step, which is to study the infinitesimal $p$-adic deformation theory of the Galois representation associated with a particular $p$-stabilisation $E_{1, \chi}^{(p)}$ of $E_{1,\chi}$. This is a $p$-adic \emph{cusp form}, and one may construct an explicit infinitesimal cuspidal Hida family $E^{(p)}_{1,\chi}(\epsilon)$ specialising to $E^{(p)}_{1,\chi}$. Mimicking the analytic proof of Gross--Zagier in \cite{GZ}, the formula for the norm was a consequence of the fact that 
\begin{equation}\label{1linepf}
f_{\text{GD}} \colonequals e^{\text{ord}}\left( \frac{d}{d\epsilon} E^{(p)}_{1,\chi}(z,z; \epsilon) \right) \in \mathcal{S}_2( \Gamma_0(N) ),
\end{equation}
must vanish. We make two remarks regarding this work.

The first is that the proof presented in \cite{daas1} excluded the case that $N = 22$ and $p = 2$, in which case the quaternion algebra $B$, which now has discriminant 11, is not of class number 1. Even though in \cite{Giam}, the authors did not present experimental evidence for their conjecture in this case, as their method of $p$-adically approximating the infinite product from Equation \ref{maineq} relied on the class number 1 assumption, they nonetheless conjectured that the value from Equation \ref{maineq} should still be algebraic, with its norm as prescribed. The complications caused by the existence of various non-conjugate maximal orders also affected the first step in the proof of \cite{daas1} sketched above, as we will explain later in detail.

The second is that the proof in \cite{daas1} showed that, in fact, an obstruction to the algebraicity of the value from Equation \ref{maineq} should not be the mere existence of cusp forms in $\mathcal{S}_2( \Gamma_0(N) )$, as Equation \ref{1linepf} might suggest, but rather the existence of a $U_p$-eigenform in this space with eigenvalue $-1$, suggesting that the original conjecture by Giampietro and Darmon could be true in this level of, or even greater, generality. 

The purpose of this work is to address both of these remarks. To state our main result, we should introduce some more notation. As we will recall in Section \ref{qFsec}, for $i \in \{ 1,2 \}$ the class group $\text{Pic}(K_i)$ acts on the set of $\Gamma$-conjugacy classes of embeddings $\alpha_i : \mathcal{O}_i \xhookrightarrow{} \mathcal{R}$ and by extension also on the set of CM points of discriminant $D_i$ on the curve $\Gamma \setminus \mathcal{H}_p$. Then, if we let $\pi \in \overline{\mathcal{R}}^{\times} - \Gamma$ be any element, we define
\[
\Theta( D_1, D_2 ) \colonequals \prod_{ \substack{ c_1 \in \text{Pic}(K_1) \\ c_2 \in \text{Pic}(K_2) } } \frac{\Theta(c_2 \cdot \tau_2, c_2 \cdot \tau_2'; c_1 \cdot \tau_1)}{\Theta( c_2 \cdot \tau_2, c_2 \cdot \tau_2'; c_1 \cdot \tau_1')},
\]
and similarly $\Theta_p( D_1, D_2 )$ with $\tau_2$ replaced by $\pi \tau_2$. Recall that $\text{w}_N$ denotes the appropriate Atkin--Lehner involution acting on $X_N$. The following is our main result, and is a direct extension of both the main conjecture from \cite{Giam} and the main theorem from \cite{daas1}.

\begin{theorem}\label{mainthm}
Suppose that the genus of $X_N / \emph{w}_N$ is zero. Then
\[
\left( \frac{\Theta( D_1, D_2 )}{\Theta_p( D_1, D_2 )} \right)^{\frac{\pm 2}{w_1w_2}} = \pm \prod_{\substack{x^2 < D \\ x^2 \equiv D \emph{ mod } 4N}} \mathfrak{F}\left( \frac{D-x^2}{4N} \right)^{\delta(x)}.
\]
\end{theorem}
Here the function $\mathfrak{F} : \mathbb{N} \to \mathbb{N}$ is an elementary function whose value on any positive integer is either $1$ or a prime power. Its definition can be found in either \cite{GZ} or \cite{daas1}; we will give an alternative description which will be more convenient for our purposes in Remark \ref{Fdef}. In this sense, the right hand side of the above formula yields a prime factorisation for the expression on the left hand side.  

Further, the value $\delta(x) \in \{ \pm 1 \}$ is a simple sign, and equal to $1$ if $x \equiv \pm a \mod 2N$ and $-1$ if $x \equiv \pm b \mod 2N$, where $\pm a, \pm b \mod 2N$ are the four square roots of $D$ modulo $4N$.

\begin{remark}\label{tableopm}
In \cite{daas1}, Theorem \ref{mainthm} was proved under the simultaneous assumptions that $B$ is of class number 1 and that $g(X_N) = 0$, which, as stated before, means that $N \in \{ 6, 10, 22 \}$ and $M = N / p$ is not equal to $11$. Therefore, our Theorem \ref{mainthm} is strictly stronger than the main result of \cite{daas1}. From the tables in \cite{oana}, it follows that the genus of $X_N / w_N$ is zero if and only if $N$ is one of
\[
\big\{ 6, 10, 14, 15, 21, 22, 26, 33, 34, 35, 38, 39, 46, 51, 55, 62, 69, 74, 86, 87, 94, 95, 111, 119, 134, 146, 159, 194, 206 \big\}.
\]
The genus of the curve $X_N$ for the values we are considering could be as high as $4$. The primes $p \mid N$ that we will be considering are now part of the set
\[
\big\{ 2, 3, 5, 7, 11, 13, 17, 19, 23, 29, 31, 37, 43, 47, 53, 67, 73, 97, 103 \big\},
\]
and therefore the class number of the quaternion algebra $B$ could be as high as $9$.  
\end{remark}

\renewcommand{\baselinestretch}{0.97}\normalsize

\begin{remark}
In the archimedean setting of \cite{crabit}, the reason for the vanishing of $f_{\text{GD}} \in \mathcal{S}_2( \Gamma_0(N) )$ for $N \in \{ 6, 10, 22 \}$ is the equation $\text{w}_N f_{\text{GD}} = f_{\text{GD}}$, combined with the non-existence of non-zero $\text{w}_N$-invariant cusp forms in this space. This is reflected by Theorem \ref{mainthm}, and using the fact that all divisors of the form $P - \text{w}_p P - \text{w}_q P + \text{w}_N P$ on $X_N$ are principal if $g( X_N / \text{w}_N ) = 0$, one may analogously define an archimedean quantity whose norm in this generality will equal the left hand side from Theorem \ref{mainthm}. From the $p$-adic perspective, the obstruction is instead given by cusp forms in the $-1$ eigenspaces for both $U_p$ and $U_q$, where $N = pq$. Proposition \ref{no-1} will show that these two conditions are in fact equivalent. This enriches the parallels with the classical work of Gross--Kohnen--Zagier \cite{GKZ}, where the height pairings between Heegner points are computed on the Jacobian of $X_0(N) / \text{w}_N$, as opposed to that of the modular curve $X_0(N)$ itself. 
\end{remark}

\begin{remark}
It is tempting to conjecture that the left hand side of Theorem \ref{mainthm} is algebraic, or rational, \emph{if and only if} the genus of $X_N / \text{w}_N$ is zero. This, together with the results from Section \ref{genussec}, would explain why one cannot expect the statement from Theorem \ref{mainthm} to hold without modification beyond the cases we consider here. So, in this sense, Theorem \ref{mainthm} is expected to be as strong as it can be without adjusting the quantity of interest or invoking the action of Hecke operators.
\end{remark}

\begin{remark}\label{thetaalg}
An even more subtle question would be to ask when the values from Equation \ref{maineq} themselves are already algebraic, before taking the quotient with its companion quantity, in which $\tau_2$ is replaced by $\pi \tau_2$. In \cite{Giam}, it was observed that in the case of $N = 34$ and $p =17$, the values of Equation \ref{maineq} are still algebraic, and their conjectured formula for the norm still seemed valid. Numerical experiments indicate that the same happens for $N = 15$ and $p = 3$. However, one fails to recognise the value of Equation \ref{maineq} as an algebraic number in the case of $N = 15$ and $p = 5$, whereas the quotient by its companion quantity seems algebraic, as predicted by Theorem \ref{mainthm}. The unique newform $f \in \mathcal{S}_2( \Gamma_0( 15 ) )$ satisfies $U_3 f = -f$ and $U_5 f = f$, suggesting a connection between these eigenvalues and the algebraicity of the values from Equation \ref{maineq}.
\end{remark}

To prove Theorem \ref{mainthm}, we first overcome the complications caused by the existence of multiple non-conjugate maximal orders in $B$ by working over $\mathbb{Z}[1/p]$ instead of integrally. To this end, in Section \ref{seqsec}, we will prove a $p$-inverted version of a well-known exact sequence from genus theory, and in Section \ref{qFsec}, we explore in detail the properties of various $F$-quadratic forms $q_F : B \to F$. 

As we will explain in Section \ref{HYsec}, the counting problem from Equation \ref{rewrite} that we are to solve has been previously considered by Howard and Yang in \cite{HY}. In Chapter 4 of the author's PhD thesis \cite{thesis}, the result from Howard and Yang was made explicit using a concrete bijection, under the assumption of a unique maximal order in $B$. We generalise this bijection to fit the $p$-inverted setting in Section \ref{Zpsec}, which is valid in all definite rational quaternion algebras $B$. In fact, we obtain a more refined result in Theorem \ref{bestcount} than what we will ultimately need, which is Corollary \ref{finalcount}. This may support future refinements of results about the values from Equation \ref{maineq} beyond their algebraic norms.

We explore various equivalent conditions for the genus of $X_N / \text{w}_N$ being zero in Section \ref{genussec}. Finally, in Section \ref{thetasec}, we will use the newly obtained generality in which we can rewrite Equation \ref{maineq} to go beyond the Giampietro--Darmon conjecture, and explain its connection to the genus of $X_N / \text{w}_N$. 
For this, we will rely strongly on the results obtained in \cite{daas1}, and we will recall what we will need from this work.

\begin{remark}\label{monster}
In \cite{pDV1}, Darmon and Vonk proposed a $p$-adic analogue in the RM setting of the differences between singular moduli as studied in \cite{GZ}; a certain \emph{rigid meromorphic cocycle} for the group $\text{SL}_2( \mathbb{Z}[1/p] )$, whose special values conjecturally display for real quadratic fields factorisations of similar intricacy to those considered hitherto. The \emph{monstrous} primes, which are those primes $p$ for which the genus of the quotient $X_0(p) / \text{w}_p$ is zero, play a key role in this work. Here $X_0(p)$ denotes the modular curve of level $\Gamma_0(p)$, and the list of monstrous primes is finite and given by $\{ 2, 3, 5, 7, 11, 13, 17, 19, 23, 29, 31, 41, 47, 59, 71 \}$. 

Interestingly, as explained in Remark \ref{genusbound}, if the genus of the quotient $X_0(N) / \text{w}_N$ is zero, then so is the genus of $X_N / \text{w}_N$ and Theorem \ref{mainthm} applies. Concretely, this happens for $N \in \{ 6, 10, 14, 15, 21, 26, 35, 39 \}$.

The analogy with the RM setting is further enriched by recent work by Damm-Johnsen \cite{damm}, in which the $F$-quadratic form $q_F$ that we will study in Section \ref{qFsec} is considered in the RM setting intead. Specifically, we refer the reader to its Section 3.1 for more details.
\end{remark}
\begin{remark}
Recently, the constructions from Darmon and Vonk were generalised to more general orthogonal groups associated with rational quadratic spaces in \cite{RMC}. If the rational quadratic space in question is of real signature $(r,s)$, then the theory predicts arithmetically rich rigid meromorphic $s$-cocycles for various $p$-arithmetic groups. From this perspective, values as those in Equation \ref{maineq} can be viewed as the special values of a rigid meromorphic $0$-cocycle, as $0$-cocycles in this context are simply $\Gamma$-invariant functions.
\end{remark}

\renewcommand{\baselinestretch}{1.00}\normalsize

\section{A $p$-inverted exact sequence}\label{seqsec}

Recall that throughout, we fix two distinct imaginary quadratic fields $K_1$ and $K_2$ with maximal orders $\mathcal{O}_1$ and $\mathcal{O}_2$ respectively. We let $D_1, D_2 < 0$ denote their discriminants and let $L = \mathbb{Q}( \sqrt{D_1}, \sqrt{D_2}) = K_1K_2$ be the biquadratic compositum of $K_1$ and $K_2$. There is a quadratic subfield of $L$ which is totally real, and if we set $D = D_1 D_2 > 0$, we will denote this field by $F = \mathbb{Q}( \sqrt{D} )$. The maximal orders in the fields $F$ and $L$ are denoted by $\mathcal{O}_F$ and $\mathcal{O}_L$ respectively.

We will assume that $\text{gcd}( D_1, D_2 ) = 1$. Then the field extension $L/F$ is unramified at all finite places, and therefore it induces a genus character $\chi \colon \text{Pic}(F)^+ \to \{ \pm 1 \}$, where $\text{Pic}(F)^+$ denotes the narrow class group of $F$. Finally, for any subset $X \subset F$, we let $X^+ \subset X$ denote the subset of totally positive elements of $X$. 

The following exact sequence was used by Gross and Zagier in Section 6 of \cite{GZ} as part of their analytic proof of the formula for the norm of the difference between two singular moduli, where the exactness of the sequence was attributed to the work of Hasse. For a modern account of it, see Section 4.3 of \cite{thesis}.

\begin{theorem}\label{exactseqold}
The following sequence is exact:
\[
1 \to \mathcal{O}_1^{\times} \mathcal{O}_2^{\times} \to \mathcal{O}_L^{\times} \xrightarrow{ \emph{Nm}^L_F } \mathcal{O}_F^{\times, +} \xrightarrow{ \varphi } \emph{Pic}(K_1) \times \emph{Pic}(K_2) \xrightarrow{\emph{ext}} \emph{Pic}(L) \xrightarrow{ \emph{Nm}^L_F } \emph{Pic}(F)^+ \xrightarrow{ \chi } \{ \pm  1 \} \to 1.
\]
\end{theorem}

All the maps in the above sequence are self-explanatory, with the exception of the map 
\[
\varphi : \mathcal{O}_F^{\times, +} \to \text{Pic}(K_1) \times \text{Pic}(K_2),
\]
whose definition relies on genus theory. Exactness on the left of the map $\varphi$ is elementary, whereas exactness on the right is a consequence of class field theory. 

Let $\epsilon_F$ denote the fundamental unit of $F$. Then $\epsilon_F \gg 0$ as a result of the existence of the unramified CM extension $L$ of $F$, and so $\mathcal{O}_F^{\times, +} = \langle \epsilon_F \rangle$. As $\text{Nm}^L_F( \epsilon_F ) = \epsilon_F^2$, the cokernel of the norm map $\text{Nm}^L_F : \mathcal{O}_L^{\times} \to \mathcal{O}_F^{\times, +}$ is of size at most 2. We relate this cokernel to the kernel of the extension map,
\[
\text{Cap} \colonequals \text{ker}\big( \text{ext} : \text{Pic}(K_1) \times \text{Pic}(K_2) \to \text{Pic}( L ) \big),
\]
which consists of those pairs of ideal classes whose product \emph{capitulates} in $L$. The key input is analytic; between the zeta-functions of the fields in our field diagram, we have the relation
\[
\zeta_L(s) \cdot \zeta(s)^2 = \zeta_{K_1}(s) \cdot \zeta_{K_2}(s) \cdot \zeta_F(s),
\]
where $\zeta = \zeta_{\mathbb{Q}}$ is the Riemann zeta-function. If $\epsilon_L \in \mathcal{O}_L^{\times}$ denotes a fundamental unit and $h_1, h_2, h_F, h_F^+$ and $h_L$ denote the sizes of the groups $\text{Pic}(K_1)$, $\text{Pic}(K_2)$, $\text{Pic}(F)$, $\text{Pic}(F)^+$ and $\text{Pic}(L)$ respectively, then from the analytic class number formula, using that $h_F^+ = 2h_F$ as $\epsilon_F \gg 0$, we obtain
\[
2|\log|\epsilon_L|| h_L = |\log|\epsilon_F||h_1h_2h_F \implies \# \text{Cap} = \frac{h_1h_2h_F^+}{2h_L} = 2\frac{|\log|\epsilon_L||}{|\log|\epsilon_F| |} \in \{ 1,2 \}.
\]
It follows that the size of the cokernel of the norm map $\text{Nm}^L_F : \mathcal{O}_L^{\times} \to \mathcal{O}_F^{\times, +}$ is equal to $\# \text{Cap}$. 

There are therefore two cases to consider. If both of these groups are trivial, then $\varphi$ is the zero map. If not, then we must send $\epsilon_F \in \mathcal{O}_F^{\times}$ to the unique pair of ideals in $\text{Cap}$. Genus theory describes this pair. 
 
As $\epsilon_F \gg 0$, its norm is 1 and Hilbert 90 ensures the existence of some $y_F \in \mathcal{O}_F$ such that $y_F / \sigma( y_F ) = \epsilon_F$, where $\sigma$ denotes the non-trivial element in $\text{Gal}( F / \mathbb{Q} )$. An elementary argument shows that one can choose such an element $y_F$, uniquely up to sign, in such a way that its norm divides $D = D_1 D_2$. Let $I_1 \subset \mathcal{O}_1$ be the ideal constructed as the product of the prime ideals in $\mathcal{O}_1$ above the (ramified) rational primes dividing the number $\text{gcd}(D_1, \text{Nm}(y_F))$, and similarly $I_2$. Then one can show that $(I_1, I_2)$ describes the non-trivial class in Cap. Whence we set $\varphi( \epsilon_F ) = ([I_1], [I_2]) \in \text{Pic}(K_1) \times \text{Pic}(K_2)$.  

Now let $p$ be a prime that is inert in both $K_1$ and $K_2$. It then splits into two primes $\mathfrak{p}$ and $\mathfrak{p}'$ in $F$, which must then both be inert in $L$ and extend to primes $\mathfrak{P}$ and $\mathfrak{P}'$. Since $\chi( [\mathfrak{p}] ) = -1$, the class $[ \mathfrak{p} ] \in \text{Pic}(F)^+$ must be of even order, say $2k$. Let $\beta \in \mathcal{O}_F^+$ be such that $\mathfrak{p}^{2k} = \beta \mathcal{O}_F$. Then
\[
\mathcal{O}_F[1/p]^{\times, +} = \langle \epsilon_F \rangle \times \langle p \rangle \times \langle \beta \rangle.
\]
We will be working with the index 2 subgroup
\[
\mathcal{O}_F[1/p]^* = \langle \epsilon_F \rangle \times \langle p^2 \rangle \times \langle \beta \rangle = \left\{ u \in \mathcal{O}_F[1/p]^{\times, +} \mid \text{ord}_{\mathfrak{p}}(u) \equiv \text{ord}_{\mathfrak{p}'}(u) \equiv 0 \mod 2 \right\} \subset \mathcal{O}_F[1/p]^{\times, +}.
\]
The purpose of this section is to prove the following $p$-inverted version of Theorem \ref{exactseqold} above:
\[
1 \to \mathcal{O}_1^{\times} \mathcal{O}_2^{\times} \to \mathcal{O}_L[1/p]^{\times} \xrightarrow{ \text{Nm}^L_F } \mathcal{O}_F[1/p]^{*} \xrightarrow{ \varphi } \text{Pic}(K_1) \times \text{Pic}(K_2) \xrightarrow{\text{ext}} \text{Pic}(\mathcal{O}_L[1/p]) \xrightarrow{ \text{Nm}^L_F } \text{Pic}(\mathcal{O}_F[1/p])^+ \to 1.
\]
We will prove exactness on both sides of the connecting map $\varphi$ first, before pasting the two together.
\begin{lemma}\label{exact12}
The following sequence is exact:
\[
1 \to \mathcal{O}_1^{\times} \mathcal{O}_2^{\times} \to \mathcal{O}_L[1/p]^{\times} \xrightarrow{\emph{Nm}^L_F} \mathcal{O}_F[1/p]^{*}.
\]
\end{lemma}
\begin{proof}
Exactness at $\mathcal{O}_1^{\times}\mathcal{O}_2^{\times}$ is clear, so we reduce to checking exactness at $\mathcal{O}_L[1/p]^{\times}$. Clearly $\mathcal{O}_1^{\times} \mathcal{O}_2^{\times} \subset \text{ker}( \text{Nm}^L_F )$. For the other inclusion, let $u \in \mathcal{O}_L[1/p]^{\times}$. Because $\mathfrak{p}$ is inert in $L/F$, we find that $\text{ord}_{\mathfrak{p}}\left( \text{Nm}^L_F( u ) \right) = 2\text{ ord}_{\mathfrak{P}}( u )$, and similarly for $\mathfrak{p}'$ and $\mathfrak{P}'$. First, this shows that indeed $\text{Nm}^L_F( \mathcal{O}_L[1/p]^{\times} ) \subset \mathcal{O}_F[1/p]^*$. Second, we find that $\text{Nm}^L_F( u ) = 1$ implies that $u \in \mathcal{O}_L^{\times}$, and therefore $u \in \mathcal{O}_1^{\times} \mathcal{O}_2^{\times}$ by Theorem \ref{exactseqold}.
\end{proof}

\begin{proposition}\label{exact22}
The following sequence is exact:
\[
\emph{Pic}(K_1) \times \emph{Pic}(K_2) \xrightarrow{\emph{ext}} \emph{Pic}(\mathcal{O}_L[1/p]) \xrightarrow{\emph{Nm}^L_F} \emph{Pic}(\mathcal{O}_F[1/p])^+ \to 1.
\]
\end{proposition}
\begin{proof}
As the two primes in $F$ and $L$ above $p$ are each other's inverses in the class groups, it follows that
\[
\text{Pic}( \mathcal{O}_L[1/p] ) = \text{Pic}(L) / \langle [\mathfrak{P}] \rangle \quad \text{and} \quad \text{Pic}( \mathcal{O}_F[1/p] )^+ = \text{Pic}(F)^+ / \langle [\mathfrak{p}] \rangle.
\]
Since $\chi( [\mathfrak{p}] ) = -1$, the two cosets for the kernel of $\chi$ inside $\text{Pic}(F)^+$ are identified in the quotient $\text{Pic}(F)^+ / \langle [\mathfrak{p}] \rangle$. Surjectivity of the norm now follows from Theorem \ref{exactseqold}. For exactness at $\text{Pic}(\mathcal{O}_L[1/p])$, we first observe that $\text{Pic}(K_1) \times \text{Pic}(K_2)$ is still contained in the kernel, so we reduce to showing the other inclusion. 

Let $C \in \text{Pic}( L )$ be in the kernel of the norm map with image in $\text{Pic}(F)^+ / \langle [\mathfrak{p}] \rangle$. Then for some $s \in \mathbb{Z}$, we must have that $\text{Nm}^L_F(C) = [\mathfrak{p}^s] \in \text{Pic}(F)^+$. Note that $s$ must be even, since $\chi \circ \text{Nm}^L_F = 1$, whereas $\chi( [\mathfrak{p}^s] ) = (-1)^s$. Write $s = 2t$, and note that because also $\text{Nm}^L_F( \mathfrak{P}^t ) = \mathfrak{p}^s$, the class of $C - [\mathfrak{P}^t] \in \text{Pic}(L)$ is contained in the kernel of the norm map $\text{Pic}(L) \to \text{Pic}(F)^+$. Therefore by Theorem \ref{exactseqold},
\[
C = ( C - [\mathfrak{P}^t] ) + [\mathfrak{P}^t] \in \text{ext}\big( \text{Pic}(K_1) \times \text{Pic}(K_2) \big) + \langle [\mathfrak{P}] \rangle, 
\]
and hence $\text{Pic}(K_1) \times \text{Pic}(K_2)$ surjects onto the kernel in question.
\end{proof}

\begin{lemma}\label{pidords}
Recall that $2k \colonequals \emph{ord}( [ \mathfrak{p} ] ) \in \emph{Pic}(F)^+$. Then $\emph{ord}( [ \mathfrak{P} ] ) \in \emph{Pic}(L)$ is equal to either $k$ or $2k$.
\end{lemma}
\begin{proof}
Let $t$ denote the order of $\mathfrak{P}$ in $\text{Pic}(L)$. Then $\text{Nm}^L_F( \mathfrak{P}^t ) = \mathfrak{p}^{2t}$ must be a principal ideal, and therefore $2k \mid 2t$, so $k \mid t$. Conversely, if $\mathfrak{p}^{2k}$ is principal, then so is $\mathfrak{P}^{2k} = \mathfrak{p}^{2k}\mathcal{O}_L$, and therefore $t \mid 2k$. These two division relations together imply that $t \in \{ k, 2k \}$, completing the proof.
\end{proof}
Let $\text{Cap}_p \colonequals \text{ker}\big( \text{ext} : \text{Pic}(K_1) \times \text{Pic}(K_2) \to \text{Pic}(\mathcal{O}_L[1/p] ) \big)$. We now define a total of four possible cases:

\renewcommand{\arraystretch}{1.25}
\begin{table}[h]
\begin{tabular}{|c||c|c|}
\hline
Case & $\epsilon_F \in \text{Nm}^L_F( \mathcal{O}_L^{\times} )$ & $\epsilon_F \notin \text{Nm}^L_F( \mathcal{O}_L^{\times} )$ \\ \hline \hline
$\text{ord}( [\mathfrak{P}] ) = k$ & Case \textbf{AA} & Case \textbf{BA} \\ \hline
$\text{ord}( [\mathfrak{P}] ) = 2k$ & Case \textbf{AB} & Case \textbf{BB} \\ \hline
\end{tabular}
\end{table}
\renewcommand{\arraystretch}{1}
\begin{proposition}\label{connectcard}
It holds that
\begin{align*}
\emph{coker}\left( \emph{Nm}^L_F : \mathcal{O}_L[1/p]^{\times} \to \mathcal{O}_F[1/p]^* \right) &\cong ( \mathbb{Z} / 2 \mathbb{Z} )^b \cong \emph{Cap}_p,
\end{align*}
where $b$ denotes the number of \textbf{B}'s in the name of the case.
\end{proposition}
\begin{proof}
Note that
\[
\mathcal{O}_L[1/p]^{\times} = \mathcal{O}_1^{\times}\mathcal{O}_2^{\times} \times \langle \epsilon_L \rangle \times \langle p \rangle \times \langle \alpha \rangle,
\]
where $\alpha \mathcal{O}_L = \mathfrak{P}^t$ and where $t \in \{ k, 2k \}$ denotes the order of $[\mathfrak{P}]$ in $\text{Pic}(L)$. Then
\[
\text{Nm}^L_F( \mathcal{O}_L[1/p]^{\times} ) = \langle \text{Nm}^L_F(\epsilon_L) \rangle \times \langle p^2 \rangle \times \langle \text{Nm}^L_F( \alpha ) \rangle.
\]
It follows that
\[
\text{coker}\left( \text{Nm}^L_F : \mathcal{O}_L[1/p]^{\times} \to \mathcal{O}_F[1/p]^* \right) = ( \langle \epsilon_F \rangle \times \langle \beta \rangle ) / ( \langle \text{Nm}^L_F( \epsilon_L ) \rangle \times   \langle \text{Nm}^L_F( \alpha ) \rangle ),
\]
where $\beta \mathcal{O}_F = \mathfrak{p}^{2k}$. If $t = k$, we may first choose $\alpha$ and then set $\beta = \text{Nm}^L_F( \alpha )$. If $t = 2k$, we first choose $\beta$ and observe that we may then set $\alpha = \beta$, so that $\text{Nm}^L_F( \alpha ) = \beta^2$. The first isomorphism follows.

For the second, we first compute the cardinality of $\text{Cap}_p$. Recall from Theorem \ref{exactseqold} that
\[
\# \text{Cap} = \frac{h_1h_2h_F^+}{2h_L} = [ \mathcal{O}_F^{\times,+} : \text{Nm}^L_F( \mathcal{O}_L^{\times} ) ].
\]
By Proposition \ref{exact22}, the cardinality of $\text{Cap}_p$ is given by
\[
\frac{h_1h_2h_F^+ / 2k}{h_L / t} = \frac{ h_1h_2h_F^+ }{ 2h_L } \cdot \frac{t}{k} = [ \mathcal{O}_F^{\times,+} : \text{Nm}^L_F( \mathcal{O}_L^{\times} ) ] \cdot \frac{t}{k}.
\]
We find that $\text{Cap}_p$ is a finite group of the claimed cardinality. That it is a 2-group only needs justification in case \textbf{BB}, but since the classes in the kernel of ext are 2-torsion in both the cases \textbf{AB} and $\textbf{BA}$, this holds true in case \textbf{BB} as well and the proof is complete.
\end{proof}
\begin{theorem}\label{exactseq}
There is an exact sequence
\[
1 \to \mathcal{O}_1^{\times} \mathcal{O}_2^{\times} \to \mathcal{O}_L[1/p]^{\times} \xrightarrow{ \emph{Nm}^L_F } \mathcal{O}_F[1/p]^{*} \xrightarrow{ \varphi } \emph{Pic}(K_1) \times \emph{Pic}(K_2) \xrightarrow{\emph{ext}} \emph{Pic}(\mathcal{O}_L[1/p]) \xrightarrow{ \emph{Nm}^L_F } \emph{Pic}(\mathcal{O}_F[1/p])^+ \to 1.
\]
\end{theorem}
\begin{proof}
The existence of a map $\varphi : \mathcal{O}_F[1/p]^{*} \to \text{Pic}(K_1) \times \text{Pic}(K_2)$ that connects both sequences from Propositions \ref{exact12} and \ref{exact22} is immediate from Proposition \ref{connectcard}, but we will be more explicit. 

If $\epsilon_F \notin \text{Nm}( \mathcal{O}_L^{\times} )$, then we send $\epsilon_F$ to the pair of ideal classes $([I_1], [I_2])$ as in the original sequence. 

If $\beta \notin \text{Nm}( \mathcal{O}_L[1/p]^{\times} )$, which happens if and only if $t = 2k$, then we send $\beta$ to any two classes $c_1 \in \text{Pic}(K_1)$ and $c_2 \in \text{Pic}(K_2)$ such that $\text{ext}(c_1,c_2) = [\mathfrak{P}^k] \in \text{Pic}(L)$, which are well-defined up to the pair $(I_1, I_2)$.

The complex property is now clear by construction, and exactness at $\text{Pic}(K_1) \times \text{Pic}(K_2)$ is checked as before, as the class of $\mathfrak{P}^k$ is trivial precisely when it has to be, establishing exactness of the full sequence.
\end{proof}

\section{An $F$-quadratic form}\label{qFsec}

Recall that we fixed a maximal $\mathbb{Z}[1/p]$-order $\mathcal{R} \subset B$, where $B$ is a definite rational quaternion algebra unramified at $p$. The following crucial result enables all that follows.
\begin{proposition}\label{cl1invp}
Every left-$\mathcal{R}$-ideal is principal, and so is every right-$\mathcal{R}$-ideal. In addition, any two maximal $\mathbb{Z}[1/p]$-orders in $B$ are conjugate.
\end{proposition} 
\begin{proof}
Since $B$ is split at $p$ by assumption, and $\text{Pic}( \mathbb{Z}[1/p] ) = 0$, this follows directly from Theorem 28.5.5 and Corollary 28.5.6 in \cite{voight}.
\end{proof}
Suppose that there exist embeddings $\alpha_1 : K_1 \xhookrightarrow{} B$ and $\alpha_2 : K_2 \xhookrightarrow{} B$. Then the images $\alpha_1( \mathcal{O}_1[1/p] )$ and $\alpha_2( \mathcal{O}_2[1/p] )$ are contained in some maximal $\mathbb{Z}[1/p]$-order of $B$. By Proposition \ref{cl1invp} above, we may conjugate these embeddings by an element in $B^{\times}$ to ensure that
\[
\alpha_1( \mathcal{O}_1[1/p] ) \subset \mathcal{R} \quad \text{and} \quad \alpha_2( \mathcal{O}_2[1/p] ) \subset \mathcal{R}.
\]
Henceforth, we will assume that this is true and fix these embeddings once and for all. 

This choice of embeddings equips $B$ with the structure of an $L$-vector space as follows. Let $x \in K_1$ and $y \in K_2$. Then the action of the element $xy \in L = K_1K_2$ on some $b \in B$ is defined by $(xy) * b \colonequals \alpha_2( y ) b \alpha_1( x )$. We may extend this definition $\mathbb{Q}$-linearly to make sense of $z * b$ for all $z \in L$. 

Note that this way, as $\mathcal{O}_L = \mathcal{O}_1 \otimes_{\mathbb{Z}} \mathcal{O}_2$, the order $\mathcal{R} \subset B$ becomes an $\mathcal{O}_L[1/p]$-module.

\begin{lemma}\label{qFexist}
There exists a unique $F$-quadratic form $q_F : B \to F^+ \cup \{ 0 \}$ with the property that
\[
\emph{tr}( q_F( b ) ) = \emph{Nm}( b ) \quad \text{for all $b \in B$.}
\]
\end{lemma}
\begin{proof}
This is Proposition 2.3 in \cite{HY}, rephrased in the language of quaternions. Alternatively, see Proposition 3.2 in \cite{damm}, or Proposition 4.4.3 in the author's PhD thesis \cite{thesis}. That the values of $q_F$ are totally positive is easily deduced from the fact that the norm form on $B$ is positive definite.
\end{proof}

For $i \in \{ 1,2 \}$, let $\text{Emb}( \mathcal{O}_i, \mathcal{R}, \Gamma )$ denote the set of $\Gamma$-conjugacy classes of embeddings $\mathcal{O}_i \xhookrightarrow{} \mathcal{R}$. Note that it makes no difference whether we consider embeddings of $\mathcal{O}_i$ or $\mathcal{O}_i[1/p]$, so to simplify notation, we will work with the former. The sets $\text{Emb}( \mathcal{O}_i, \mathcal{R}, \Gamma )$ are in bijection with the sets of CM points of discriminants $D_1$ and $D_2$ respectively on the curve $\Gamma \setminus \mathcal{H}_p$, and they carry an action of the class groups $\text{Pic}(K_i)$ as follows. 

Let $0 \neq J_1 \subset \mathcal{O}_1$ be an integral ideal. Then the left-$\mathcal{R}$-ideal $\mathcal{R} \alpha_1( J_ 1 )$ is principal, so there exists $\xi_1 \in \mathcal{R}$ such that $\mathcal{R} \alpha_1( J_ 1 ) = \mathcal{R} \xi_1$. The ambiguity in $\xi_1$ is measured by $\mathcal{R}^{\times}$. After possibly multiplying on the left by any $\varpi \in \mathcal{R}^{\times}$ for which $\text{ord}_p( \text{Nm}( \varpi ) )$ is odd, we may assume that $\text{ord}_p( \text{Nm}( \xi_1 ) )$ is even. We then set
\[
J_1 \cdot \alpha_1(-) \colonequals \xi_1 \alpha_1(-) \xi_1^{-1} : \mathcal{O}_1 \xhookrightarrow{} \mathcal{R}.
\]
Similarly, for an ideal $0 \neq J_2 \subset \mathcal{O}_2$, the set $\alpha_2( J_2 ) \mathcal{R}$ is a fractional right-ideal of $R$. Again, this ideal must be principal and therefore there exists an element $\xi_2 \in \mathcal{R}$ such that $\alpha_2( J_2 ) \mathcal{R} = \xi_2 \mathcal{R}$ and such that $\text{ord}_p( \text{Nm}( \xi_2 ) )$ is even. We then set
\[
J_2 \cdot \alpha_2(-) \colonequals \xi_2^{-1} \alpha_2(-) \xi_2 : \mathcal{O}_2 \xhookrightarrow{} \mathcal{R}.
\]
It is not obvious from this definition that these recipes actually define group actions of the class groups of the fields $K_i$ on the sets $\text{Emb}( \mathcal{O}_i, \mathcal{R}, \Gamma )$ for $i \in \{ 1,2 \}$. The following proposition aims to remedy this; for a more general description of this action in the integral context, we refer the reader to Section 3 of \cite{grossheights}.

\begin{proposition}\label{picpf}
Let $0 \neq J \subset \mathcal{O}_1$ be an ideal. Then the construction above defines a group action of $\emph{Pic}(K_1)$ on the set $\emph{Emb}( \mathcal{O}_1, \mathcal{R}, \Gamma )$, and similarly for the action of $\emph{Pic}(K_2)$ on the set $\emph{Emb}( \mathcal{O}_2, \mathcal{R}, \Gamma )$.
\end{proposition}
\begin{proof}
We will prove this for the action of $\text{Pic}(K_1)$, as the proof for the action of $\text{Pic}(K_2)$ is similar. To improve notation and avoid overloading, we mostly avoid using the subscript 1 in the proof below.

We first show that the $\Gamma$-conjugacy class of the embedding $J \cdot \alpha$ is independent of the choice of $\xi$. Indeed, suppose that $\xi' \in  \mathcal{R}$ is also such that $\mathcal{R} \xi' = \mathcal{R} \alpha( J ) = \mathcal{R} \xi$ and that $\text{ord}_p( \text{Nm}( \xi' ) )$ is even. Then $\mathcal{R} \xi' \xi^{-1} = \mathcal{R}$ and therefore $\xi' \xi^{-1} \in \mathcal{R}$. Similarly, we show that $\xi \xi'^{-1} \in \mathcal{R}$. Therefore $\xi' \xi^{-1} \in \mathcal{R}^{\times}$ and as $\text{ord}_p( \text{Nm}( \xi' \xi^{-1} ) )$ is even, up to some factors of $p$, we have $\xi' \xi^{-1} \in \Gamma$ and the claim is proved.

We continue by showing that the embedding $J \cdot \alpha$ takes values in $\mathcal{R}$. This is because $\mathcal{R} \alpha( J )$ is a right-$\mathcal{O}_1$-module, and so $\mathcal{R} \xi \alpha( \mathcal{O}_1 ) = \mathcal{R} \xi$. Therefore $\mathcal{R} \xi \alpha( \mathcal{O}_1 ) \xi^{-1} = \mathcal{R}$ and so $\xi \alpha( \mathcal{O}_1 ) \xi^{-1} \subset \mathcal{R}$. 

Next, if $x \in K_1^{\times}$, then $(Jx) \cdot \alpha = J \cdot \alpha$ up to $\Gamma$-conjugacy. Indeed, the action of $Jx$ is described by conjugation by $\xi \alpha( x )$, because $\text{ord}_p( \text{Nm}( \alpha(x) ) ) = \text{ord}_p( \text{Nm}(x) )$ must be even, as $p$ is assumed inert in $\mathcal{O}_1$. Now, as $\alpha(x)$ commutes with the image of $\alpha$, its conjugation action is the same as conjugation by $\xi$.

It follows that $J \cdot \alpha$ only depends on the class $[J] \in \text{Pic}(K_1)$. It thus remains to show that it is in fact a group action. To this end, let $J' \subset \mathcal{O}_1$ be another ideal. To compute the action of $J'$ on $J \cdot \alpha$, we first write $\mathcal{R} \xi \alpha( J' ) \xi^{-1} = \mathcal{R} \xi'$ for some $\xi' \in \mathcal{R}$, so we find that the embedding $J' \cdot (J \cdot \alpha)$ is given by $\xi' \xi \alpha( - ) \xi^{-1} \xi'^{-1}$. Since $\mathcal{R} \xi = \mathcal{R} \alpha(J)$ by definition, we find that $\mathcal{R} \alpha( J ) \alpha( J' ) \xi^{-1} = \mathcal{R} \xi'$, and therefore $\mathcal{R} \alpha( J \cdot J' ) = \mathcal{R} \xi' \xi$. In other words, the action of $J \cdot J'$ can also be described by conjugation by $\xi' \xi$, and the proof is complete.
\end{proof}

\begin{remark}\label{nmopm}
By considering the indices inside $\mathcal{R}$ on both sides of $\mathcal{R} \alpha_1( J ) = \mathcal{R} \xi_1$ for some ideal $J \subset \mathcal{O}_1$ and $\xi_1 \in \mathcal{R}$, one sees that away from $p$, the quaternion norm of $\xi_1$ is equal to the ideal norm of $J$. As $p$ is inert in $\mathcal{O}_1$, the number of factors of $p$ in the norm of $J$ is even. Adjusting $\xi_1$ by some power of $p$, we may thus always choose $\xi_1 \in \mathcal{R}$ such that $\text{Nm}( \xi_1 ) = \text{Nm}( J )$. The same holds for the action of ideals of $\mathcal{O}_2$.
\end{remark}

It will be important to understand the effect on the quadratic form $q_F$ after conjugating our embeddings.

\begin{proposition}\label{picdetf}
Let $\alpha_1 : K_1 \xhookrightarrow{} B$ and $\alpha_2 : K_2 \xhookrightarrow{} B$ be embeddings and write $q_F : B \to F^+ \cup \{ 0 \}$ for the $F$-quadratic form they induce. Further, let $\xi_1, \xi_2 \in B^{\times}$ be arbitrary and let $\tilde{q}_F : B \to F^+ \cup \{ 0 \}$ be the $F$-quadratic form induced by $\xi_1 \alpha_1( - ) \xi_1^{-1}$ and $\xi_2^{-1} \alpha_2( - ) \xi_2$. Then
\[
\tilde{q}_F(b) = \frac{q_F( \xi_2 b \xi_1 )}{\emph{Nm}(\xi_2 \xi_1)} \quad \text{for all $b \in B$}.
\]
\end{proposition}
\begin{proof}
It suffices to show that the claimed formula satisfies the two defining properties of $\tilde{q}_F$. First,
\[
\text{tr} \left( \frac{q_F( \xi_2 b \xi_1 )}{\text{Nm}(\xi_2\xi_1)} \right) = \frac{\text{Nm}( \xi_2 b \xi_1 )}{\text{Nm}(\xi_2\xi_1)} = \text{Nm}(b)
\]
for all $b \in B$. To show it is $F$-quadratic, we denote the action of $L$ on $B$ through $\alpha_1$ and $\alpha_2$ by $*$, whereas the action through $\xi_1 \alpha_1( - ) \xi_1^{-1}$ and $\xi_2^{-1} \alpha_2( - ) \xi_2$ is denoted by $\star$. Now for $\sqrt{D} = \sqrt{D_1} \cdot \sqrt{D_2} \in F$,
\begin{align*}
\tilde{q}_F(\sqrt{D} \star b) &= \tilde{q}_F(\xi_2^{-1} \alpha_2( \sqrt{D_2} ) \xi_2 b \xi_1 \alpha_1( \sqrt{D_1} ) \xi_1^{-1} ) = \frac{q_F( \alpha_2( \sqrt{D_2} ) \xi_2 b \xi_1 \alpha_1( \sqrt{D_1} ) )}{\text{Nm}(\xi_2\xi_1)} \\
&= \frac{q_F( \sqrt{D} * ( \xi_2 b \xi_1 ) )}{\text{Nm}(\xi_2\xi_1)} = D \frac{q_F( \xi_2 b \xi_1 )}{\text{Nm}(\xi_2\xi_1)} = D \tilde{q}_F(b).
\end{align*}
The statement for general elements of $F$ is now an easy consequence.
\end{proof}

For any pair of ideal classes $c_1 \in \text{Pic}(K_1)$ and $c_2 \in \text{Pic}(K_2)$, we let $q_F[c_1,c_2] : B \to F^+ \cup \{ 0 \}$ denote the $F$-quadratic form associated with any pair of choices of representatives for the $\Gamma$-conjugacy classes of the embeddings $c_1 \cdot \alpha_1$ and $c_2 \cdot \alpha_2$. This leaves quite some ambiguity in the definition of the form $q_F[c_1,c_2]$, but as we will mostly be concerning ourselves with the cardinalities of sets of the form
\[
\# \{ b \in \mathcal{R} \mid q_F[c_1,c_2](b) = \nu \}
\]
for some $\nu \in F^+$, the following remark is reassuring. 
\begin{remark}\label{qfnotdef}
Given $\xi_1, \xi_2 \in \Gamma$, it holds that $\text{Nm}( \xi_1 ) = 1 = \text{Nm}( \xi_2 )$, and therefore by Proposition \ref{picdetf},
\begin{align*}
\# \{ b \in \mathcal{R} \mid \tilde{q}_F(b) = \nu \} = \# \{ b \in \mathcal{R} \mid q_F(\xi_2 b \xi_1 ) = \nu \} = \# \{ b \in \mathcal{R} \mid q_F( b ) = \nu \},
\end{align*}
since the association $b \mapsto \xi_2 b \xi_1$ is a bijection on $\mathcal{R}$ as $\xi_1, \xi_2 \in \Gamma \subset \mathcal{R}^{\times}$. In other words, the solutions to these counting problems are independent of the choice of representative embeddings in their respective $\Gamma$-conjugacy classes, and the ambiguity of the forms $q_F[c_1,c_2]$ will not interfere with our goals.
\end{remark}

Special care needs to be taken when considering the action of the pairs in $\text{Cap}_p \subset \text{Pic}(K_1) \times \text{Pic}(K_2)$ that capitulate in $\text{Pic}(\mathcal{O}_L[1/p])$. This is established in the following two propositions. We first need a lemma.

\begin{lemma}\label{actequal}
Let $J_1 \subset \mathcal{O}_1$ and $J_2 \subset \mathcal{O}_2$ be ideals such that $J_1\mathcal{O}_L[1/p] \cdot J_2 \mathcal{O}_L[1/p] = y \mathcal{O}_L[1/p]$ for some $y \in \mathcal{O}_L[1/p]$. Then
\[
y^{-1} * ( \alpha_2(J_2) \mathcal{R} \alpha_1( J_1 ) ) = \mathcal{R}.
\]
\end{lemma}
\begin{proof}
It suffices to show that the left hand side is included in the right hand side, for the same argument with all players involved inverted will show the other inclusion.

Taking the norm down to $F$, it follows that $\text{Nm}(J_1)\text{Nm}(J_2) = \text{Nm}^L_F( y ) \mathcal{O}_F[1/p]$. Let $\tau \in \text{Gal}(L / F)$ denote the non-trivial element. Acting by $\text{Nm}^L_F( y ) = y \cdot \tau(y)$, to prove the lemma it suffices to show that
\[
\tau(y) * ( \alpha_2(J_2) \mathcal{R} \alpha_1( J_1 ) ) \subset \text{Nm}^L_F( y ) * \mathcal{R} = \text{Nm}(J_1) \text{Nm}(J_2) \mathcal{R}.
\]
To see this, we note that for any $x_1 \in J_1$ and $x_2 \in J_2$,
\[
\tau(x_1) \tau(x_2) * ( \alpha_2(J_2) \mathcal{R} \alpha_1( J_1 ) ) = \alpha_2( \tau(x_2) J_2 ) \mathcal{R} \alpha_1( \tau(x_1) J_1 ) \subset \text{Nm}(J_1) \text{Nm}(J_2) \mathcal{R},
\]
as $J_i \cdot \tau( J_i ) = \text{Nm}(J_i) \mathcal{O}_i$ for $i \in \{ 1, 2 \}$. Now as $\tau(J_1) \mathcal{O}_L[1/p] \cdot \tau(J_2) \mathcal{O}_L[1/p] = \tau(y) \mathcal{O}_L[1/p]$ and $\mathcal{R}$ is a $\mathcal{O}_L[1/p]$-module, the proof is complete.
\end{proof}

\begin{proposition}\label{idealambig2}
In cases \textbf{BA} and \textbf{BB}, let $(c_1,c_2) = \varphi( \epsilon_F ) \in \emph{Pic}(K_1) \times \emph{Pic}(K_2)$. Then
\[
q_F[c_1,c_2](\mathcal{R}) = \epsilon_F \cdot q_F(\mathcal{R}).
\]
\end{proposition}
\begin{proof}
Let $I_1 \subset \mathcal{O}_1$ and $I_2 \subset \mathcal{O}_2$ be the two ideals used in the definition of $\varphi( \epsilon_F )$. We recall that the ideals $I_1$ and $I_2$ are supported only at ramified primes and that the product of their norms equals the norm of an element $y_F \in \mathcal{O}_F$ that satisfies $y_F / \sigma( y_F ) = \epsilon_F$, where $\sigma$ denotes the non-trivial element in $\text{Gal}( F / \mathbb{Q} )$. 

Let $\xi_1 \in \mathcal{R}$ be such that $\mathcal{R} \alpha_1(I_1) = \mathcal{R} \xi_1$, so that $c_1 \cdot \alpha_1 = \xi_1 \alpha_1( - ) \xi_1^{-1}$. Similarly, let $\xi_2 \in \mathcal{R}$ be such that $ \alpha_2(I_2) \mathcal{R} =  \xi_2 \mathcal{R}$, so that $c_2 \cdot \alpha_2 = \xi_2^{-1} \alpha_2( - ) \xi_2$. By Proposition \ref{picdetf}, the form $q_F[c_1,c_2]$ is given by
\[
q_F[c_1,c_2](b) = \frac{q_F( \xi_2 b \xi_1 )}{\text{Nm}(\xi_2 \xi_1)}.
\]
By Remark \ref{nmopm}, we may assume that $\text{Nm}( \xi_1 ) = \text{Nm}( I_1)$ and $\text{Nm}( \xi_2 ) = \text{Nm}( I_2 )$. Now for $b \in B$, set
\[
b' = y_F^{-1} * (\xi_2 b \xi_1).
\]
Using that $q_F$ is $F$-quadratic and that $\text{Nm}( y_F ) = \text{Nm}(I_1) \cdot \text{Nm}(I_2) = \text{Nm}(\xi_1\xi_2)$, we compute that
\begin{align*}
\epsilon_F q_F(b') = \frac{y_F}{\sigma(y_F)} q_F\left( y_F^{-1} * (\xi_2 b \xi_1 ) \right) = \frac{1}{y_F \sigma( y_F )} q_F( \xi_2 b \xi_1 ) = \frac{\text{Nm}(\xi_1\xi_2)}{\text{Nm}(y_F)} q_F[c_1,c_2](b) = q_F[c_1,c_2](b).
\end{align*}
Now by Lemma \ref{actequal}, it holds that $b' \in \mathcal{R}$ if and only if $b \in \mathcal{R}$, and the proof is complete.
\end{proof}

\begin{proposition}\label{idealambig3}
In cases \textbf{AB} and \textbf{BB}, let $(c_1,c_2) = \varphi( \beta ) \in \emph{Pic}(K_1) \times \emph{Pic}(K_2)$. Then
\[
q_F[c_1,c_2](\mathcal{R}) = \beta \cdot q_F(\mathcal{R}).
\]
\end{proposition}
\begin{proof}
Let $c_1 = [J_1] \in \text{Pic}(K_1)$ and $c_2 = [J_2] \in \text{Pic}(K_2)$. By definition of $\varphi( \beta )$, we know that $c_1 \cdot c_2 = [\mathfrak{P}^k] \in \text{Pic}(L)$ where $2k$ denotes the order of $[\mathfrak{p}]$ in $\text{Pic}(F)^+$, and therefore we can find $y \in L$ such that $J_1 \mathcal{O}_L \cdot J_2 \mathcal{O}_L = (y) \mathfrak{P}^k$. Taking the norm down to $F$ yields that $\text{Nm}(J_1J_2) \mathcal{O}_F = ( \text{Nm}^L_F(y) ) \mathfrak{p}^{2k} = ( \text{Nm}^L_F(y) \beta )$. Exploiting the ambiguity in our definition of $\beta$ up to $\mathcal{O}_F^{\times, +}$, we may assume that $\text{Nm}(J_1J_2) = \text{Nm}^L_F(y) \beta$.

We adopt the same notation as in the proof of Proposition \ref{idealambig2} above, and for $b \in B$ we now set
\[
b' = (y \beta )^{-1} * (\xi_2 b \xi_1).
\]
We now compute (implicitly using Remark \ref{isoL}) that
\begin{align*}
\beta q_F(b') = \beta q_F\left( (y \beta)^{-1} * (\xi_2 b \xi_1) )\right) = \frac{1}{ \beta \text{Nm}^L_F( y )} q_F( \xi_2 b \xi_1 ) = \frac{\text{Nm}( \xi_1 \xi_2 )}{\text{Nm}(J_1J_2)} q_F[c_1,c_2](b) = q_F[c_1,c_2](b).
\end{align*}
The proof is again completed by invoking Lemma \ref{actequal}.
\end{proof}

\section{The work of Howard--Yang}\label{HYsec}

We recall the work of Howard and Yang from \cite{HY}, who set out to find a geometric interpretation of the values of the function $\mathfrak{F}$ that occur both in the factorisation of the differences between singular moduli from \cite{GZ}, as well as in our Theorem \ref{mainthm}. As part of their analysis, they solved a counting problem very similar to the one described by Equation \ref{rewrite}. We recall their setup and result.

For $i \in \{ 1,2 \}$, let $\textbf{E}_i = ( E_i, \kappa_i )$ be a fixed elliptic curve over $\overline{\mathbb{Q}}$ with complex multiplication by $\mathcal{O}_i$, where $\kappa_i : \mathcal{O}_i \xrightarrow{\sim} \text{End}( E_i )$ is one of two possible ring isomorphisms. The class group $\text{Pic}(K_i)$ acts simply transitively on the set of elliptic curves $E_i$ with CM by $\mathcal{O}_i$, and for an ideal $\mathfrak{a}_i \subset \mathcal{O}_i$, we denote this action by $\textbf{E}_i \otimes \mathfrak{a}_i$. Up to isomorphism, this action only depends on the class $[\mathfrak{a}_i]$ in the class group $\text{Pic}(K_i)$. As there are two choices for the morphism $\kappa_i$ for any given elliptic curve, it follows that for $i \in \{ 1,2 \}$, the number of elliptic curves $\textbf{E}_i = (E_i, \kappa_i)$ with CM by $\mathcal{O}_i$ is given by $2h_i$, where $h_i = \# \text{Pic}(K_i)$ is the class number of $K_i$.

Fix a rational prime $q$ and assume that it is inert in both $\mathcal{O}_1$ and $\mathcal{O}_2$. By CM theory, all elliptic curves with CM by $\mathcal{O}_1$ or $\mathcal{O}_2$ have supersingular reduction at $q$. We fix a prime $\mathfrak{Q}$ above $q$ in $\overline{ \mathbb{Q} }$ and we let $\overline{\textbf{E}}_i$ denote the reduction of $\textbf{E}_i$ modulo $\mathfrak{Q}$ for $i \in \{ 1,2 \}$.

Given a pair $(\textbf{E}_1, \textbf{E}_2)$ as above, for $i \in \{ 1,2 \}$, the action of $\mathcal{O}_i$ through $\kappa_i$ on the Lie algebra $\text{Lie}(E_i)$, or equivalently, on the 1-dimensional space of invariant differentials on $E$, induces a ring morphism $\mathcal{O}_i \xhookrightarrow{} \overline{\mathbb{Q}}$. As $\mathcal{O}_L = \mathcal{O}_1 \otimes_{\mathbb{Z}} \mathcal{O}_2$, this data determines a ring morphism $\mathcal{O}_L \xhookrightarrow{} \overline{ \mathbb{Q} }$. We let $\mathfrak{q}$ denote the prime of $F$ below the prime of $L$ defined by the preimage of $\mathfrak{ Q }$ under this map. This \emph{reflex prime} is invariant under the action of $\text{Pic}(K_1) \times \text{Pic}(K_2)$, but changes to the other prime $\mathfrak{q}'$ above $q$ in $F$ when replacing any of the $\kappa_i$ by its Galois conjugate embedding.

As the curves $\overline{\textbf{E}}_i$ for $i \in \{ 1, 2 \}$ are supersingular, for any pair of ideals $\mathfrak{a}_i \subset \mathcal{O}_i$ for $i \in \{ 1,2 \}$, the space $\text{Hom}\left( \textbf{E}_1 \otimes \mathfrak{a}_1, \textbf{E}_2 \otimes \mathfrak{a}_2 \right)$ is a free $\mathbb{Z}$-module of rank 4. Precomposition by the image of $\mathcal{O}_1$ under $\kappa_1$ and postcomposition by the image of $\mathcal{O}_2$ under $\kappa_2$ turns this space naturally into an $(\mathcal{O}_2, \mathcal{O}_1)$-bimodule. This group comes with the positive definite quadratic degree form deg. The space $V(\textbf{E}_1, \textbf{E}_2) = \text{Hom}(\textbf{E}_1, \textbf{E}_2) \otimes \mathbb{Q}$ is now a $(K_2,K_1)$-bimodule and therefore the quadratic degree form induces a unique $F$-quadratic form
\[
\text{deg}_{\text{CM}} : V(\textbf{E}_1, \textbf{E}_2) \to F^+ \cup \{ 0 \}.
\]
such that $\text{tr} \circ \text{deg}_{\text{CM}} = \text{deg}$; see our Lemma \ref{qFexist} or Proposition 2.3 in \cite{HY}. If we let $\mathcal{D}_F$ denote the different ideal of $F$ and write $w_i = \# \mathcal{O}_i^{\times}$ for $i \in \{ 1, 2 \}$, then the following is proved in \cite{HY} if one combines their Proposition 2.18 and Corollary 2.34.
\begin{theorem}[Howard--Yang]\label{HYthm}
Fix $\textbf{E}_1$ and $\textbf{E}_2$ as above and let $\nu \in F^+$. The number of triples
\[
( \phi, [\mathfrak{a}_1], [\mathfrak{a}_2] ) \in \emph{Hom}(\textbf{E}_1 \otimes \mathfrak{a}_1, \textbf{E}_2 \otimes \mathfrak{a}_2) \times \emph{Pic}(K_1) \times \emph{Pic}(K_2)
\]
satisfying $\emph{deg}_{\emph{CM}}( \phi ) = \nu$ is equal to $\frac{w_1w_2}{2}\rho( \nu \mathcal{D}_F \mathfrak{q}^{-1} )$, where for an ideal $I \subset \mathcal{O}_F$,
\[
\rho( I ) \colonequals \# \{ J \subset \mathcal{O}_L \mid \emph{Nm}^L_F( J )  = I \}.
\]
\end{theorem}

In \cite{HY}, this counting problem is tackled using an ad\`elic approach, and one observes that its result is not subject to any restrictions on the class or type number of the rational quaternion algebra $B$ ramified only at the prime $q$. The purpose of the next section is to describe an explicit bijection between the two sets whose cardinalities are related in Theorem \ref{HYthm}, which has the pleasant property of being global and not ad\`elic in nature. However, we will need to invert $p$ to combat the complications introduced by the existence of non-conjugate maximal orders in the quaternionic formulation.

\begin{remark}
An equally clean but purely quaternionic restatement of Theorem \ref{HYthm} using the Deuring correspondence is problematic for the following reason. Let $\textbf{E} / \overline{ \mathbb{F} }_q$ be a supersingular elliptic curve. Then its geometric endomorphism ring $\text{End}(\textbf{E})$ is isomorphic to a maximal order $R$ in $B$, and one may choose this isomorphism in such a way so as to identify the degree form on one side with the norm for on the order. However, if $\textbf{E}' / \overline{ \mathbb{F} }_q$ is another supersingular elliptic curve whose endomorphism ring may be identified with a different, non-conjugate maximal order $R'$ in $B$, then there is in general no canonical identification between $\text{Hom}(\textbf{E}, \textbf{E}')$ and a lattice in $B$ that identifies the degree form with the norm form; this can only be done up to some rational scalar. Any purely quaternionic reformulation of Theorem \ref{HYthm} would rely on making non-canonical choices for these lattices, which include additional ad-hoc correction factors to account for the difference between these two quadratic forms. The resulting statement is then not immediately useful.
\end{remark}

\section{An explicit $p$-inverted bijection}\label{Zpsec}

In view of Theorem \ref{mainthm} and the values that appear in Remark \ref{tableopm}, it suffices for our purposes to specialise to rational quaternion algebras $B$ that are ramified at a single prime. As in the previous section, we will let $q$ denote this prime and assume that $q \neq p$. As the embeddings $\alpha_1 : \mathcal{O}_1 \xhookrightarrow{} \mathcal{R}$ and $\alpha_2 : \mathcal{O}_2 \xhookrightarrow{} \mathcal{R}$ equip $B$ with the structure of a 1-dimensional $L$-vector space, we may choose an isomorphism
\[
\iota : B \xrightarrow{\sim} L
\]
of $L$-vector spaces. 
\begin{remark}\label{isoL}
One can show that any choice of such an isomorphism $\iota : B \xrightarrow{\sim} L$ will identify the $F$-quadratic form $q_F$ on $B$ with the $F$-quadratic form $\beta \cdot \text{Nm}^L_F$ on $L$ for some $\beta \in F^{\times}$ that depends on $\iota$; see Lemma 2.5 in \cite{HY} or Proposition 3.3 in \cite{damm}. In particular, it follows that $q_F( (xy) * b ) = \text{Nm}(x) \text{Nm}(y) q_F(b)$ for all $x \in K_1$ and $y \in K_2$. In particular, if $x \in \mathcal{O}_1^{\times}$ and $y \in \mathcal{O}_2^{\times}$, then $q_F( (xy) * b ) = q_F(b)$.
\end{remark}
Define for any $b \in B$ the ideal
\[
I_b \colonequals \iota(b) \cdot \iota(\mathcal{R})^{-1} \subset L.
\]
\begin{lemma}\label{iotaindep2}
The set $\iota(\mathcal{R})$ defines a fractional $\mathcal{O}_L[1/p]$-ideal in $L$. For any $b \in B$, the ideal $I_b$ is independent of the choice of isomorphism $\iota : B \xrightarrow{\sim} L$ and $I_b$ is $\mathbb{Z}[1/p]$-integral if and only if $b \in \mathcal{R}$. 
\end{lemma}
\begin{proof}
It is easy to see that $\iota( \mathcal{R} )$ is finitely generated as an $\mathcal{O}_L[1/p]$-module and as it is a submodule of $L$, it is a fractional $\mathcal{O}_L[1/p]$-ideal. To see that $I_b$ is $\mathbb{Z}[1/p]$-integral if and only if $b \in \mathcal{R}$, we note that
\[
I_b = \iota(b) \cdot \iota(\mathcal{R})^{-1} \subset \mathcal{O}_L[1/p] \iff \iota(b) \in \iota(\mathcal{R}) \mathcal{O}_L[1/p] = \iota(\mathcal{R}) \iff b \in \mathcal{R}.
\]
Finally, to see the independence of $I_b$ from $\iota$, we note that any two isomorphisms of 1-dimensional $L$-vector spaces agree up to a scalar, so any other $\iota'$ can be written as $\lambda \cdot \iota$ for some $\lambda \in L^{\times}$. Then indeed
\[
\iota'(b) \iota'( \mathcal{R} )^{-1} = \lambda \iota(b) \cdot \lambda^{-1} \iota( \mathcal{R} )^{-1} = \iota( b ) \iota( \mathcal{R} ),
\]
and the proof is complete.
\end{proof}

We wish to compute the norm to $F$ of the ideal $I_b$ for any $b \in B$.

\begin{proposition}\label{detFatell2}
Let $\ell \neq p,q$ be a prime number. Then for any $b \in B$,
\[
\emph{Nm}^L_F( I_b ) \otimes \mathbb{Z}_{\ell} = ( q_F(b) \mathcal{D}_{F} ) \otimes \mathbb{Z}_{\ell}.
\]
\end{proposition}
\begin{proof}
For $\ell \neq p$, after completing at $\ell$, there is no difference between maximal $\mathbb{Z}$-orders and maximal $\mathbb{Z}[1/p]$-orders. One may now appeal to Lemma 2.16 in \cite{HY}.
\end{proof}

As in \cite{HY}, we define the \emph{reflex ideal} of $\mathcal{O}_F$ above the rational prime $q$ associated with the pair of embeddings $( \alpha_1, \alpha_2 )$; see also Subsection 42.4.6 in \cite{voight}. The ring morphisms
\[
\alpha_i : \mathcal{O}_i \xhookrightarrow{} \mathcal{R} \to \mathcal{R} / [\mathcal{R}, \mathcal{R}] \cong \mathbb{F}_{q^2}
\]
for $i \in \{ 1,2 \}$ give rise to a map $\mathcal{O}_L = \mathcal{O}_1 \otimes_{\mathbb{Z}} \mathcal{O}_2 \to \mathbb{F}_{q^2}$. We let the reflex ideal be the $\mathcal{O}_F$-ideal
\[
\mathfrak{q} \colonequals \mathcal{O}_F \cap \text{ker}\left( \mathcal{O}_L \to \mathbb{F}_{q^2} \right).
\]

\begin{proposition}\label{detFatq2}
Let $b \in B$. Then
\[
\emph{Nm}^L_F( I_b ) \otimes \mathbb{Z}_{q} = ( q_F(b) \mathfrak{q}^{-1} ) \otimes \mathbb{Z}_{q}.
\]
\end{proposition}
\begin{proof}
Again, as $q \neq p$, we may appeal to Lemma 2.22 in \cite{HY}.
\end{proof}

These two results combine to prove the following key result.

\begin{corollary}\label{idealnorm2}
As $\mathcal{O}_F[1/p]$-ideals, for any $b \in B$, the ideal $I_b$ satisfies
\[
\emph{Nm}^L_F(I_b) = q_F(b) \mathcal{D}_F \mathfrak{q}^{-1}.
\]
\end{corollary}
\begin{proof}
This follows from the above two results as the ideal norm may be computed locally.
\end{proof}

\begin{remark}\label{evenval}
The above shows for any unramified prime $\mathfrak{l} \nmid pq$ of $F$ which is inert in $L$, that the $\mathfrak{l}$-adic valuation of $q_F(b)$ must be even, as this is so for all ideals that are norms from $L$. Repeating this argument with any rational prime $\ell \neq p$ that satisfies the same hypotheses shows that this also holds for $\mathfrak{p}, \mathfrak{p}' \subset \mathcal{O}_F$.
\end{remark}

We next take into account the actions of the class groups on the embeddings. Given $c_1 \in \text{Pic}(K_1)$ and $c_2 \in \text{Pic}(K_2)$, we let $\iota[c_1,c_2] : B' \xrightarrow{\sim} L$ be an isomorphism of $L$-vector spaces, where $B' = B$ is equipped with the $L$-vector space structure induced by the embeddings $c_1 \cdot \alpha_1$ and $c_2 \cdot \alpha_2$. We will also consider
\[
I[c_1,c_2]_b \colonequals \iota[c_1,c_2](b) \cdot \iota[c_1,c_2]( \mathcal{R} )^{-1}.
\] 
We invite the reader to compare this next lemma and its proof to Proposition \ref{picdetf}.
\begin{lemma}\label{piciota}
Let $\xi_1, \xi_2 \in B^{\times}$ and let $B'$ denote the $L$-vector space with underlying set $B$ but with $L$-vector space structure induced by the embeddings $\xi_1 \alpha_1 \xi_1^{-1}$ and $\xi_2^{-1} \alpha_2 \xi_2$. Then there is an isomorphism of $L$-vector spaces
\[
\iota' : B' \to L : b \mapsto \iota( \xi_2 b \xi_1 ).
\]
\end{lemma}
\begin{proof}
Clearly $\iota'$ is a $\mathbb{Q}$-linear bijective map of sets, so it suffices to verify its $L$-linearity. Let $*$ denote the $L$-action on $B$ and $\star$ the $L$-action on $B'$. Then for $x \in K_1$ and $y \in K_2$, we compute that
\begin{align*}
\iota'\left( (xy) \star b \right) &= \iota'\left( \xi_2^{-1} \alpha_2(y) \xi_2 b \xi_1 \alpha_1(x) \xi_1^{-1} \right) = \iota\left( \alpha_2(y) \xi_2 b \xi_1 \alpha_1(x) \right) \\
&= \iota\left( (xy) * ( \xi_2 b \xi_1 ) \right) = xy \cdot \iota\left( \xi_2 b \xi_1 \right) = xy \cdot \iota'(b),
\end{align*}
and the proof is complete.
\end{proof}

\begin{proposition}\label{picact2}
Let $c_1 \in \emph{Pic}(K_1)$ and $c_2 \in \emph{Pic}(K_2)$. Then the ideal class of $I[c_1,c_2]_b$ inside of $\emph{Pic}(\mathcal{O}_L[1/p])$ is given by $[I_b] - c_1 - c_2$, where we regard $c_1,c_2 \in \emph{Pic}(\mathcal{O}_L[1/p])$ through the extension map \emph{ext}. 
\end{proposition}
\begin{proof}
Let $J_1 \subset \mathcal{O}_1$ and $J_2 \subset \mathcal{O}_2$ be two ideals representing the classes $c_1$ and $c_2$ respectively and let $\xi_1, \xi_2 \in \mathcal{R}$ be such that $\mathcal{R} \alpha_1( J_1 ) = \mathcal{R} \xi_1$ and $\alpha_2( J_2 ) \mathcal{R} = \xi_2 \mathcal{R}$, so that $J_1 \cdot \alpha_1( - ) = \xi_1 \alpha_1( - ) \xi_1^{-1}$ and $J_2 \cdot \alpha_2( - ) = \xi_2^{-1} \alpha_2( - ) \xi_2$. By Lemma \ref{piciota} above, we have $\iota[c_1,c_2](b) = \iota(\xi_2 b \xi_1 )$ for all $b \in B$. We now compute that
\[
\iota[c_1,c_2]( \mathcal{R} ) = \iota( \xi_2 \mathcal{R} \xi_1 ) = \iota( \alpha_2( J_2 ) \mathcal{R} \alpha_1( J_1 ) ) = (J_1 \cdot J_2) \cdot \iota( \mathcal{R} ),
\]
and the proof is complete, as the scalars $\iota(b)$ and $\iota[c_1,c_2](b)$ will not change the classes of $I_b$ and $I[c_1,c_2]_b$.
\end{proof}

We are now ready to prove the main results that we have been working towards. Recall that
\[
\text{Cap}_p \colonequals \text{ker}\big( \text{ext} : \text{Pic}(K_1) \times \text{Pic}(K_2) \to \text{Pic}( \mathcal{O}_L[1/p] ) \big)
\]
is a subgroup of capitulating pairs of size 1, 2 or 4 as described by Proposition \ref{connectcard}. Finally, let 
\[
C_0 \colonequals [ \iota( \mathcal{R} )^{-1} ] \in \text{Pic}( \mathcal{O}_L[1/p] ).
\]
\begin{theorem}\label{refinedthm}
Let $\nu \in F^+$ be such that $\emph{ord}_{\mathfrak{p}}( \nu )$ and $\emph{ord}_{\mathfrak{p}'}( \nu )$ are both even. Then the association
\[
(b,c_1,c_2) \mapsto I[c_1,c_2]_b
\]
establishes a $\frac{w_1w_2}{2}$-to-1 mapping between the sets
\begin{align*}
&\{ (b,c_1,c_2) \in \mathcal{R} \times \emph{Cap}_p \mid q_F[c_1,c_2](b) = \nu \} \text{ and} \\
&\{ \mathcal{I} \subset \mathcal{O}_L[1/p] \mid \emph{Nm}^L_F( \mathcal{I} ) = ( \nu ) \mathcal{D}_F \mathfrak{q}^{-1} \text{ and } [\mathcal{I}] = C_0 \in \emph{Pic}( \mathcal{O}_L[1/p] ) \}.
\end{align*}
\end{theorem}
\begin{proof}
First, for any triple $(b, c_1, c_2) \in \mathcal{R} \times \text{Pic}(K_1) \times \text{Pic}(K_2)$, Lemma \ref{iotaindep2} shows that $I[c_1,c_2]_b$ is an integral $\mathcal{O}_L[1/p]$-ideal. Proposition \ref{idealnorm2} shows that the $\mathcal{O}_F[1/p]$-ideal $\text{Nm}^L_F( I[c_1,c_2]_b )$ is equal to $q_F( b ) \mathcal{D}_F \mathfrak{q}^{-1}$. Finally, the class of $I[c_1,c_2]_b$ in $\text{Pic}(\mathcal{O}_L[1/p])$ is $C_0$ by Proposition \ref{picact2} and the definition of $\text{Cap}_p$. Combined, this shows that the mapping is well defined. We determine the cardinality of each pre-image.

To this end, let $\mathcal{I} \subset \mathcal{O}_L[1/p]$ be an ideal such that $\text{Nm}^L_F( \mathcal{I} ) = (\nu) \mathcal{D}_F \mathfrak{q}^{-1}$ and $[ \mathcal{I} ] = C_0 \in \text{Pic}( \mathcal{O}_L[1/p] )$. Then the class of $\mathcal{I} \cdot \iota( \mathcal{R} )$ is trivial in $\text{Pic}( \mathcal{O}_L[1/p] )$. As $\iota : B \to L$ is bijective, there exists some $b \in B$ such that $\mathcal{I} \cdot \iota( \mathcal{R} ) = (\iota(b))$, and therefore $\mathcal{I} = I_b$. As $\mathcal{I}$ is integral, it follows from Lemma \ref{iotaindep2} that $b \in \mathcal{R}$.

Using Proposition \ref{idealnorm2}, taking the norm to $F$ we find $(\nu) \mathcal{D}_F \mathfrak{q}^{-1} = q_F( b ) \mathcal{D}_F \mathfrak{q}^{-1}$ as $\mathcal{O}_F[1/p]$-ideals. Since $q_F(b)$ and $\nu$ are both totally positive, it follows that $q_F(b) = \epsilon \cdot \nu$ for some $\epsilon \in \mathcal{O}_F[1/p]^{\times, +}$. From Remark \ref{evenval} in combination with our assumption on $\nu$, it follows that $\epsilon \in \mathcal{O}_F[1/p]^*$. Recall from Remark \ref{isoL} that for $u \in L$, it holds that $q_F( u * b ) = \text{Nm}^L_F( u ) q_F(b)$. We proceed to distinguish no less than four cases.

\textbf{Case AA:} In this case, we can always find $u \in \mathcal{O}_L^{\times, +}$ such that $\text{Nm}^L_F( u ) = \epsilon$. It then follows that $q_F( u^{-1} * b ) = \text{Nm}^L_F( u )^{-1} q_F(b) = \epsilon^{-1} \epsilon \cdot \nu = \nu$, so we move on.

\textbf{Case AB:} Now $\mathcal{O}_F[1/p]^* / \text{Nm}^L_F( \mathcal{O}_L[1/p]^{\times} ) \cong \mathbb{Z} / 2 \mathbb{Z}$ by Proposition \ref{connectcard}, with coset representatives given by $1$ and $\beta$, which by Theorem \ref{exactseq} correspond to two different pairs of ideal classes in $\text{Cap}_p \cong \mathbb{Z} / 2 \mathbb{Z}$. By Proposition \ref{idealambig3}, if we act on our embeddings by the non-trivial class, all $q_F$-values will shift by a factor of $\beta$, whereas the ideal class of $I[c_1,c_2]_b$ in $\text{Pic}( \mathcal{O}_L[1/p] )$ does not change. It follows that there is a unique pair $(c_1,c_2) \in \text{Cap}_p$ for which we may act by a $p$-unit of $L$ to find some $b \in \mathcal{R}$ such that $q_F[c_1,c_2](b) = \nu$.

\textbf{Case BA:} Now again $\mathcal{O}_F[1/p]^* / \text{Nm}^L_F( \mathcal{O}_L[1/p]^{\times} ) \cong \mathbb{Z} / 2 \mathbb{Z}$ by Proposition \ref{connectcard}, but with coset representatives given by $1$ and $\epsilon_F$. The proof is then completed analogously, invoking Proposition \ref{idealambig2} instead.

\textbf{Case BB:} This is a combination of the two cases above.

We have now found $(b, c_1, c_2) \in \mathcal{R} \times \text{Cap}_p$ such that $q_F[c_1,c_2]( b ) = \nu$. The pair $(c_1,c_2) \in \text{Cap}_p$ is unique, and the ambiguity in our choice of $b \in \mathcal{R}$ is given by 
\[
\text{ker}\big( \text{Nm}^L_F : \mathcal{O}_L[1/p]^{\times} \to \mathcal{O}_F[1/p]^* \big) = \mathcal{O}_1^{\times} \mathcal{O}_2^{\times},
\]
and this ambiguity remains by Remark \ref{isoL}. Since $\# ( \mathcal{O}_1^{\times} \mathcal{O}_2^{\times} ) = w_1w_2 / 2$, this completes the proof.
\end{proof}

\begin{corollary}\label{refinedcount}
Let $\nu \in F^+$ be $p$-integral. Then
\[ 
\# \{ (b,c_1,c_2) \in \mathcal{R} \times \emph{Cap}_p \mid q_F[c_1,c_2](b) = \nu \}
\]
is equal to
\[
\frac{w_1w_2}{2} \# \{ J \subset \mathcal{O}_L \mid \emph{Nm}^L_F( J ) = ( \nu ) \mathcal{D}_F \mathfrak{q}^{-1} \text{ and } [ J \mathcal{O}_L[1/p] ] = C_0 \in \emph{Pic}( \mathcal{O}_L[1/p] ) \}.
\]
\end{corollary}
\begin{proof}
First, if either $\text{ord}_{\mathfrak{p}}( \nu )$ or $\text{ord}_{\mathfrak{p}'}( \nu )$ is odd, then there cannot exist such an ideal $J \subset \mathcal{O}_L$, for its norm will always have an even number of factors of $\mathfrak{p}$ and $\mathfrak{p}'$. By Remark \ref{evenval}, neither can the value $\nu$ be attained by the form $q_F[c_1,c_2]$, so the result trivially holds in this case.

On the other hand, if these valuations are both even, then using Theorem \ref{refinedthm}, it suffices to note that there will now be a unique $\mathcal{O}_L$-ideal $J$ with the right norm if we set $\text{ord}_{\mathfrak{P}}( J ) = \text{ord}_{\mathfrak{p}}( \nu ) / 2$ and $\text{ord}_{\mathfrak{P}'}( J ) = \text{ord}_{\mathfrak{p}'}( \nu ) / 2$. Indeed, $J$ is integral since $\nu$ is assumed $p$-integral.
\end{proof}

In fact, we can be even a bit more precise.

\begin{lemma}\label{Puncertain}
In cases \textbf{AA} and \textbf{BA}, there is a unique class $\widetilde{C}_0 \in \emph{Pic}(L)$ whose image in $\emph{Pic}( \mathcal{O}_L[1/p] )$ is $C_0$, and whose image under the norm map to $\emph{Pic}(F)^+$ is the class of $[ \mathcal{D}_F \mathfrak{q}^{-1} ]$. In cases \textbf{AB} and \textbf{BB}, there are two such classes, $\widetilde{C}_0$ and $\widetilde{C}_0 + [ \mathfrak{P}^k ]$, where $2k = \emph{ord}([\mathfrak{p}]) \in \emph{Pic}(F)^+$.
\end{lemma}
\begin{proof}
Under the conditions from the lemma, the uncertainty in the class $\widetilde{C}_0 \in \text{Pic}(L)$ is measured by the size of the intersection of $\text{ker}( \text{Pic}(L) \to \text{Pic}(F)^+ )$ and the subgroup $\langle [ \mathfrak{P} ] \rangle$. As $\text{Nm}^L_F( \mathfrak{P}^k ) = \mathfrak{p}^{2k}$ is the smallest power of $\mathfrak{P}$ whose norm is trivial in $\text{Pic}(F)^+$, this intersection is exactly the subgroup generated by $[ \mathfrak{P}^k ]$. By Lemma \ref{pidords}, this group is trivial or of size two precisely as described in the lemma. 
\end{proof}

In analogy with Theorem \ref{HYthm}, we set for $C \in \text{Pic}(L)$ and $I \subset \mathcal{O}_F$ the notation 
\[
\rho_C( I ) := \# \{ J \subset \mathcal{O}_L \mid \text{Nm}^L_F( J ) = I \text{ and } [ J ] = C \in \text{Pic}( L ) \}.
\]
\begin{theorem}\label{bestcount}
There is a unique $\widetilde{C}_0 \in \emph{Pic}(L)$ with image $C_0 \in \emph{Pic}( \mathcal{O}_L[1/p] )$ and for all $p$-integral $\nu \in F^+$,
\[ 
\# \{ (b,c_1,c_2) \in \mathcal{R} \times \emph{Cap} \mid q_F[c_1,c_2](b) = \nu \} = \frac{w_1w_2}{2} \rho_{\widetilde{C}_0} \! \left( (\nu) \mathcal{D}_F \mathfrak{q}^{-1} \right).
\]
\end{theorem}
\begin{proof}
In cases \textbf{AA} and \textbf{BA}, it holds that $\text{Cap} = \text{Cap}_p$. On the other hand, if an ideal $J \subset \mathcal{O}_L$ satisfies $\text{Nm}^L_F( J ) = ( \nu ) \mathcal{D}_F \mathfrak{q}^{-1}$ and $[ J \mathcal{O}_L[1/p] ] = C_0 \in \text{Pic}( \mathcal{O}_L[1/p] )$, then from Lemma \ref{Puncertain} it follows that the class $[J] = \widetilde{C}_0 \in \text{Pic}(L)$ is uniquely determined. The claim is therefore immediate from Corollary \ref{refinedcount}.

In cases \textbf{AB} and \textbf{BB}, we recall from the proof of Theorem \ref{refinedthm} that the pair $(c_1,c_2) \in \text{Cap}_p$ is constant among all triples $(b,c_1,c_2)  \in \mathcal{R} \times \text{Cap}_p$ for which $q_F[c_1,c_2](b) = \nu$. We may then conclude the proof using Corollary \ref{refinedcount} if we show that this pair $(c_1,c_2)$ is contained in $\text{Cap} \subset \text{Cap}_p$ if and only if any ideal $J \subset \mathcal{O}_L$ such that $\text{Nm}^L_F( J ) = ( \nu ) \mathcal{D}_F \mathfrak{q}^{-1}$ and $[ J \mathcal{O}_L[1/p] ] = C_0 \in \text{Pic}( \mathcal{O}_L[1/p] )$ has the property that $[ J ] = \widetilde{C}_0 \in \text{Pic}(L)$ for the appropriate choice of $\widetilde{C}_0$. 

Recall that every such ideal $J$ is obtained from the constraints $J \mathcal{O}_L[1/p] = I[c_1,c_2]_b$ and $\text{ord}_{\mathfrak{P}}( J ) = \text{ord}_{\mathfrak{p}}( \nu ) / 2$ and $\text{ord}_{\mathfrak{P}'}( J ) = \text{ord}_{\mathfrak{p}'}( \nu ) / 2$. It follows that the ideal class of $J$ in $\text{Pic}(L)$ is entirely determined by the valuations of $\nu$ at $p$. By assumption, the generator $\beta$ of $\mathfrak{p}^{2k}$ is not a norm from $\mathcal{O}_L[1/p]^{\times}$ and therefore $\text{Nm}( \mathcal{O}_L[1/p]^{\times} ) = \langle \text{Nm}( \epsilon_L ) \rangle \times \langle p^2 \rangle \times \langle \beta^2 \rangle$. By Remark \ref{isoL}, the values of $q_F[c_1,c_2]( - )$ all differ up to a norm from $L$. It follows that the difference between the valuations at $\mathfrak{p}$ and $\mathfrak{p}'$ for any such value is a multiple of $4k = v_{\mathfrak{p}}( \beta^2 )$. Therefore, the difference between the valuations at $\mathfrak{P}$ and $\mathfrak{P}'$ of any such ideal $J$ is a multiple of $2k$. But $\mathfrak{P}^{2k}$ is principal, and it follows that the class $[ J ] \in \text{Pic}(L)$ is constant for fixed $(c_1,c_2) \in \text{Cap}_p$. 

Conversely, by Propositions \ref{idealambig2} and \ref{idealambig3}, for any $(c_1, c_2) \in \text{Cap}_p - \text{Cap}$, the values of the forms $q_F$ and $q_F[c_1,c_2]$ on $\mathcal{R}$ will differ up to a factor of $\beta$ or $\epsilon_F \beta$. As $\text{ord}_{\mathfrak{p}}( \beta ) = 2k$ and $\text{ord}_{\mathfrak{p}'}( \beta ) = 0$, it follows that the ideals whose norms equal $( \nu ) \mathcal{D}_F \mathfrak{q}^{-1}$ and $( \nu \beta ) \mathcal{D}_F \mathfrak{q}^{-1}$ respectively must differ by a factor of $\mathfrak{P}^k$, and therefore they will live in the two distinct ideal classes $\widetilde{C}_0$ and $\widetilde{C}_0 + [ \mathfrak{P}^k ]$ from Lemma \ref{Puncertain}. 
\end{proof}

\begin{proposition}\label{HYthmquat4real}
Let $\nu \in F^+$ be such that $\emph{ord}_{\mathfrak{p}}( \nu )$ and $\emph{ord}_{\mathfrak{p}'}( \nu )$ are both even. Then the association
\[
(b,c_1,c_2) \mapsto I[c_1,c_2]_b
\]
establishes a $\frac{w_1w_2}{2}$-to-1 mapping between the sets
\[
\{ (b,c_1,c_2) \in \mathcal{R} \times \emph{Pic}(K_1) \times \emph{Pic}(K_2) \mid q_F[c_1,c_2](b) = \nu \} \text{ and } \{ \mathcal{I} \subset \mathcal{O}_L[1/p] \mid \emph{Nm}^L_F( \mathcal{I} ) = ( \nu ) \mathcal{D}_F \mathfrak{q}^{-1} \}.
\]
\end{proposition}
\begin{proof}
We have seen before that the map is well defined. Now let $\mathcal{ I } \subset \mathcal{O}_L[1/p]$ be any ideal such that $\text{Nm}^L_F( \mathcal{I} ) = ( \nu ) \mathcal{D}_F \mathfrak{q}^{-1}$ as $\mathcal{O}_F[1/p]$-ideals. For any triple $(b,c_1, c_2) \in \mathcal{R} \times \text{Pic}(K_1) \times \text{Pic}(K_2)$, the norm to $F$ of the ideal $I[c_1,c_2]_b$ is also in the class $[ \mathcal{D}_F \mathfrak{q}^{-1} ] \in \text{Pic}( \mathcal{O}_F[1/p] )^+$. By Theorem \ref{exactseq}, the ideal class of any $I[c_1,c_2]_b$ in $\text{Pic}(\mathcal{O}_L[1/p])$ must therefore differ up to an element from $\text{Pic}(K_1) \times \text{Pic}(K_2)$ from the class of $\mathcal{I}$. Using Proposition \ref{picact2}, this means that there exist $c_1 \in \text{Pic}(K_1)$ and $c_2 \in \text{Pic}(K_2)$ such that the class $[\mathcal{I}]$ agrees with the class $[I[c_1,c_2]_b]$ for any $b \in B^{\times}$, and the ambiguity in this choice is precisely measured by $\text{Cap}_p$. The theorem now follows from applying Theorem \ref{refinedthm} to each class in the preimage of $[ \mathcal{D}_F \mathfrak{q}^{-1} ] \in \text{Pic}( \mathcal{O}_F[1/p] )^+$ inside $\text{Pic}( \mathcal{O}_L[1/p] )$, and summing up the results. 
\end{proof}

The following is now immediate from the above, and establishes our $p$-inverted analogue of Theorem \ref{HYthm}.

\begin{corollary}\label{finalcount}
Let $\nu \in F^+$ be $p$-integral. Then
\[ 
\# \{ (b,c_1,c_2) \in \mathcal{R} \times \emph{Pic}(K_1) \times \emph{Pic}(K_2) \mid q_F[c_1,c_2](b) = \nu \} = \frac{w_1w_2}{2} \rho( ( \nu ) \mathfrak{q}^{-1}\mathcal{D}_F ).
\]
\end{corollary}

\section{The genus of $X_N / \text{w}_N$}\label{genussec}

Let $N$ be a squarefree positive integer and consider the space $\mathcal{S}_2( N ) \colonequals \mathcal{S}_2( \Gamma_0(N) )$ of weight 2 cusp forms of level $\Gamma_0(N)$. This space comes with a natural action by the Hecke algebra of level $N$, which is generated by operators $T_{\ell}$ for primes $\ell \nmid N$ and $U_{\ell}$ for primes $\ell \mid N$. In addition, it carries an action of the Atkin--Lehner group, which is generated by commuting involutions $\text{w}_{\ell}$ for primes $\ell \mid N$, called the Atkin--Lehner operators. We will need the following elementary lemma.

\begin{lemma}\label{Upnew}
Let $N$ be a positive integer and let $p \mid N$ be a prime such that $p^2 \nmid N$. Let $f \in \mathcal{S}_2(N)$ be such that $U_p f = \pm f$, Then $f$ is $p$-new.
\end{lemma}
\begin{proof}
The $p$-old space of $\mathcal{S}_2(N)$ is generated by the Hecke eigenforms from levels $d \mid N$ where $p \nmid d$. If $g$ is an eigenform of level $d$, let $\alpha$ and $\beta$ be the roots of the Hecke polynomial $X^2 - a_p(g)X + p$ of $g$ at $p$. Then $g$ admits two $p$-stabilisations, $g_{\alpha}(z) = g(z) - \beta g(pz)$ and $g_{\beta}(z) = g(z) - \alpha g(pz)$, which are $U_p$-eigenforms of level $N$ with eigenvalues $\alpha$ and $\beta$ respectively. If $\pm 1$ were is a root of this polynomial, then we would find $a_p(g) = \pm (p+1)$. But as $(\sqrt{p}-1)^2 > 0$, it would follow that $| a_p(g) | = p+1 > 2 \sqrt{p}$, contradicting the Hasse--Weil bound. 
\end{proof}

When restricted to the new subspace $\mathcal{S}^{\text{new}}_2( N ) \subset \mathcal{S}_2( N )$, it is well known that
\[
\text{w}_p + U_p = 0.
\]
Recall that the space $\mathcal{S}_2( N )$ may be identified with the space of holomorphic differentials on the modular curve $X_0(N)$. In particular, its dimension equals the genus of this curve.

Various of the following results concern only a finite number of cases, in which, in principle, one may verify the validity of the statements by brute force computation. However, we nonetheless make an effort to give proofs that are as conceptual as possible, to stress that these facts are not mere coincidences. 

\begin{proposition}\label{equiv1}
Let $N$ be a squarefree positive integer supported at an even number of primes. Then the following are equivalent:
\begin{enumerate}[label=\roman*)]
\item The genus of the curve $X_N / \emph{w}_N$ is zero;
\item If $f \in \mathcal{S}^{\emph{new}}_2(N)$ satisfies $\emph{w}_N f = f$, then $f = 0$.
\end{enumerate}
\end{proposition}
\begin{proof}
The Jacquet-Langlands correspondence establishes a Hecke-equivariant isomorphism between the space of holomorphic differentials on $X_N$ and the new subspace $\mathcal{S}^{\text{new}}_2( N )$ of $\mathcal{S}_2( N )$. However, as one finds in Theorem 1.2 in \cite{bertolini}, there is a slight sign twist; see also their Section 1.8 or Theorem 2.4 in \cite{oana}. Under this isomorphism, the $\pm 1$-eigenspace of the Atkin--Lehner operator $\text{w}_p$ for any prime $p \mid N$ on the differentials on $X_N$ is identified with the $\mp 1$-space of the Atkin--Lehner operator $\text{w}_p$ on the differentials on $\mathcal{S}^{\text{new}}_2( N )$. As $N$ is the product of an \emph{even} number of primes, it follows that the the actions of the respective operators $\text{w}_N$ on both sides may be identified. Those differentials that are fixed by $\text{w}_N$ are precisely the differentials that descend to the quotient $X_N / \text{w}_N$. Therefore, the genus of $X_N / \emph{w}_N$ being zero is equivalent to the absence of nonzero cusp forms in $\mathcal{S}^{\text{new}}_2( N )$ that are fixed by the Atkin--Lehner operator $\text{w}_N$.
\end{proof}

\begin{remark}\label{genusbound}
It follows from the above that the genus of $X_N / \text{w}_N$ is bounded above by the genus of $X_0( N ) / \text{w}_N$. Indeed, latter quantity is the dimension of the 1-eigenspace of the operator $\text{w}_N$ acting on $\mathcal{S}_2( N )$, whereas the former is the dimension of this same eigenspace when restricted to the new subspace $\mathcal{S}^{\text{new}}_2( N )$. In particular, if $X_0( N ) / \text{w}_N$ has genus zero, then so does $X_N / \text{w}_N$; see also Remark \ref{monster}.
\end{remark}

\begin{lemma}\label{Uptr}
Let $N$ be a squarefree positive integer supported at an even number of primes and let $p \mid N$. Then the trace of $U_p$ acting on $\mathcal{S}^{\emph{new}}_2(N)$ is at most 1, with equality if and only if:
\begin{itemize}
\item There exist primes $q_1, q_2 \mid N$ such that $q_1 \equiv 1 \mod 4$ and $q_2 \equiv 1,3 \mod 8$ if $p = 2$; 
\item There exists an odd prime $q \mid N$ such that $\left( \frac{-p}{q} \right) = 1$ if $p \equiv 1 \mod 4$;
\item There exists a prime $q \mid N$ such that $\left( \frac{-p}{q} \right) = 1$ if $p \equiv 3 \mod 4$.
\end{itemize}
\end{lemma}
\begin{proof}
Let $g$ denote the genus of the Shimura curve $X_N$. Then, by the above, $g = \text{dim}( \mathcal{S}^{\text{new}}_2(N) )$. As $U_p + \text{w}_p = 0$ on this space, the $\pm 1$-eigenspace of $U_p$ equals the $\mp 1$-eigenspace of $\text{w}_p$, which is in turn identified with the $\pm 1$ eigenspace of the Atkin--Lehner operators $\text{w}_p$ on the holomorphic differentials on $X_N$. Therefore, the dimension of the $+1$-eigenspace of the $U_p$-operator on $\mathcal{S}^{\text{new}}_2(N)$ is equal to the genus of the Atkin--Lehner quotient $X_N / \text{w}_p$. By the Riemann-Hurwitz formula, 
\[
g( X_N / \text{w}_p ) = (g+1)/2 - \# \{ \text{fixed points of $\text{w}_p$ on $X_N$} \} / 4.
\]
It now follows that $\text{tr}( U_p ) = 1 - \# \{ \text{fixed points of $\text{w}_p$ on $X_N$} \} / 2$, and therefore $\text{tr}(U_p) \leq 1$ with equality if and only if $\text{w}_p$ does not have any fixed points on $X_N$. As explained in Section 2 of \cite{ogg83}, such fixed points correspond to embeddings of an order in $\mathbb{Q}( \sqrt{-p} )$ into the indefinite quaternion algebra $\mathfrak{B}$ ramified at precisely those places dividing $N$, and we obtain no fixed points precisely when there are no such embeddings. The precise version of this statement in \cite{ogg83} then yields the lemma. 
\end{proof}

\begin{corollary}\label{Vcor}
Let $p \neq q$ be two prime numbers and let $N = pq$. Decompose $\mathcal{S}^{\emph{new}}_2( N )$ as the direct sum of $V_{++} \oplus V_{+-} \oplus V_{-+} \oplus V_{--}$ corresponding to the signs of the eigenspaces of the $U_p$ and $U_q$ operators respectively. Then $\emph{dim}( V_{++} ) + \emph{dim}( V_{--} ) = g( X_N / \emph{w}_N )$, and $\emph{ dim}( V_{++} ) \leq \emph{ dim} ( V_{--} ) + 1$, with equality if and only if $\emph{tr}(U_p) = \emph{tr}(U_q) = 1$ on the space $\mathcal{S}^{\emph{new}}_2( N )$.
\end{corollary}
\begin{proof}
From the reasoning in the proof of Proposition \ref{equiv1}, the Atkin-Lehner operator $\text{w}_N$ may be identified with $U_p U_q$ on $\mathcal{S}^{\text{new}}_2( N )$, and therefore the subspace $V_{++} \oplus V_{--}$ is identified with the space of differentials on the quotient $X_N / \text{w}_N$. Since $U_p$ and $U_q$ are involutions that share a common basis of eigenvectors, it follows that $2 \text{ dim}( V_{++} ) -  2 \text{ dim} ( V_{--} ) = \text{tr}( U_p ) + \text{tr}( U_q )$. By Lemma \ref{Uptr}, this is at most 2.
\end{proof}

\begin{remark}
For more general weights $k \geq 4$, Corollary 3.5 in \cite{signs} shows that sign patterns with only minus signs for the Hecke eigenforms spanning the module $\mathcal{S}^{\text{new}}_k( N )$ are, depending on $k$ and the number of primes dividing $N$, always either the most common, or the least common. For weight $k = 2$, there are complications, causing this statement to fail if $N$ has an even number of prime factors. Indeed, for most small $N = pq$, it holds that $\text{dim}(V_{++}) \leq \text{dim}(V_{--})$, but this fails first for $N = 145 = 5 \cdot 29$. 
\end{remark}

\begin{proposition}\label{no-1}
Let $p \neq q$ be two prime numbers and let $N = pq$. Then the following are equivalent:
\begin{enumerate}[label=\roman*)]
\item If $f \in \mathcal{S}^{\emph{new}}_2(N)$ satisfies $\emph{w}_N f = f$, then $f = 0$;
\item If $f \in \mathcal{S}_2(N)$ satisfies $U_p f = -f = U_q f$, then $f = 0$.
\end{enumerate}
\end{proposition}
\begin{proof}
From Lemma \ref{Upnew}, it follows that if $f \in \mathcal{S}_2(N)$ satisfies $U_p = -f$, then it must be $p$-new. Therefore, if both $U_p = -f$ and $U_q f = -f$, then $f$ must be new, and it satisfies $U_p U_q f = f$. As $\text{w}_p = - U_p$ and $\text{w}_q = -U_q$ on this space, the action of $\text{w}_N$ on $\mathcal{S}^{\text{new}}_2( N )$ coincides with the action of $U_p U_q$. Therefore, \emph{i)} implies \emph{ii)}.

Conversely, assuming \emph{ii)}, in the language of Corollary \ref{Vcor}, we are given that $V_{--} = 0$. From that same corollary, it now follows that $\text{dim}( V_{++} ) \leq 1$ with equality as described in the lemma, and it follows that $g(X_N / \text{w}_N ) \leq 1$. To conclude, in view of Proposition \ref{equiv1}, we must show that equality never holds.

It is here that the author has failed to find a conceptual way to finish the proof, but one can proceed as follows. From the tables in \cite{oana}, it follows that the genus of $X_N / \text{w}_N$ is equal to 1 if and only if $N$ is part of the list $\{ 57, 58, 65, 77, 82, 106, 118, 122, 129, 143, 166, 215, 314 \}$. However, for all these $N$ and $p \mid N$, with the sole exception of $N = 215$ and $p = 5$, one may check using Lemma \ref{Uptr} that the trace of $U_p$ on $\mathcal{S}_2^{\text{new}}( N )$ fails to be 1, and thus the proposition is proved.
\end{proof}

\begin{remark}
Let $N$ be a squarefree positive integer supported at an even number of primes and let $p \mid N$ be a prime divisor. Using analogous arguments, one can also show that the following are equivalent:
\begin{enumerate}[label=\emph{\roman*)}]
\item We have an equality of genera $g( X_N / \text{w}_p ) = g(X_N)$;
\item If $f \in \mathcal{S}^{\text{new}}_2(N)$ satisfies $U_p = -f$, then $f = 0$.
\item If $f \in \mathcal{S}_2(N)$ satisfies $U_p = -f$, then $f = 0$.
\end{enumerate}
The proof of the equivalence of \emph{i)} and \emph{ii)} is very similar to the proof of Proposition \ref{equiv1}. For the equivalence between \emph{ii)} and \emph{iii)}, one uses Lemma \ref{Upnew} to reduce to showing that the non-existence of $p$-new $U_p$-eigenforms with eigenvalue $-1$ guarantees the non-existence of $p$-old $U_p$-eigenforms with eigenvalue $-1$. Again, the author does not have a conceptual proof of this, other than relying on the finite enumeration of all cases, showing that $p \leq 23$, in which case the monstrous nature of $p$ from Remark \ref{monster} implies the claim. 

Finally, in analogy to Remark \ref{genusbound}, one may show that if $X_0( N ) / \emph{w}_p$ has genus zero, then in fact $g( X_N / \emph{w}_p ) = g(X_N)$. The author expects these observations to play a role in a potential resolution of the question raised in Remark \ref{thetaalg}, where $p$ and $q$ no longer behave symmetrically.
\end{remark}

\section{Beyond the Giampietro--Darmon conjecture}\label{thetasec}

Recall from the introduction that our main object of study is the quantity
\[
\frac{ \Theta( \tau_2, \tau_2'; \tau_1 ) }{ \Theta( \tau_2, \tau_2', \tau_1') } = \prod_{ \gamma \in \Gamma } \frac{( \tau_1 - \gamma \tau_2 )( \tau_1' - \gamma \tau_2' )}{ ( \tau_1' - \gamma \tau_2 )( \tau_1 - \gamma \tau_2') },
\]
where $\tau_1, \tau_1'$ and $\tau_2, \tau_2'$ are two pairs of Galois conjugate CM points inside the $p$-adic upper half plane $\mathcal{H}_p$. 

\begin{proposition}\label{thetawelldef}
Let $(w_1, w_2), (z_1, z_2) \in \mathcal{H}_p \times \mathcal{H}_p$ be such that the $\Gamma$-orbits of $w_1$ and $w_2$ are disjoint from the $\Gamma$-orbits of $z_1$ and $z_2$. Then the quantity
\[
\frac{ \Theta( w_1, w_2; z_1 ) }{ \Theta( w_1, w_2, z_2) } \in \mathbb{C}_p
\]
is dependent only on the $\Gamma$-orbits of the points $(w_1, w_2), (z_1, z_2)  \in \Gamma \setminus ( \mathcal{H}_p \times \mathcal{H}_p )$, with $\Gamma$ acting on $\mathcal{H}_p \times \mathcal{H}_p$ diagonally through M\"obius transformations.
\end{proposition}
\begin{proof}
Let $\delta \in \Gamma$. Then it is clear that
\[
\frac{ \Theta( \delta w_1, \delta w_2; z_1 ) }{ \Theta( \delta w_1, \delta w_2, z_2) } = \prod_{\gamma \in \Gamma} \frac{(z_1 - \gamma \delta w_1)(z_2 - \gamma \delta w_2)}{(z_1 - \gamma \delta w_2)(z_2 - \gamma \delta w_1)} = \prod_{\gamma \in \Gamma} \frac{(z_1 - \gamma w_1)(z_2 - \gamma w_2)}{(z_1 - \gamma w_2)(z_2 - \gamma w_1)} = \frac{ \Theta( w_1, w_2; z_1 ) }{ \Theta( w_1, w_2, z_2) },
\]
where we performed the bijective substitution $\gamma \leftrightarrow \gamma \delta^{-1}$. Slightly more subtly, recall the automorphy factor $c_{w_1, w_2} : \Gamma \to \mathbb{C}_p^{\times}$ associated with the Theta function $\Theta( w_1, w_2; z )$. We then compute that
\[
\frac{ \Theta( w_1, w_2; \delta z_1 ) }{ \Theta( w_1, w_2, \delta z_2) } = \frac{c_{w_1, w_2}( \delta )}{c_{w_1, w_2}( \delta )}  \frac{ \Theta( w_1, w_2; z_1 ) }{ \Theta( w_1, w_2, z_2) } =  \frac{ \Theta( w_1, w_2; z_1 ) }{ \Theta( w_1, w_2, z_2) },
\]
and the proposition is proved.
\end{proof}

As explained in the introduction, $B^{\times}$ acts on the $p$-adic upper half plane $\mathcal{H}_p$. Therefore, through the embeddings $\alpha_1, \alpha_2$, both $K_1^{\times}$ and $K_2^{\times}$ will act on $\mathcal{H}_p$ as well. As $p$ is assumed inert in both $K_1$ and $K_2$, there are two Galois conjugate fixed points $\tau_1, \tau_1' \in \mathcal{H}_p$ and $\tau_2, \tau_2' \in \mathcal{H}_p$ for these actions. As we conjugate either embedding $\alpha_i$ with some $\delta \in \Gamma$, the CM points $\tau_i, \tau_i' \in \mathcal{H}_p$ are replaced by $\delta \tau_i$ and $\delta \tau_i'$ respectively. It follows from Proposition \ref{thetawelldef} that the quantity from Equation \ref{maineq}, and therefore also the left hand side of Theorem \ref{mainthm}, is well-defined and independent of the CM points in the $p$-adic upper half plane $\mathcal{H}_p$ that we chose to represent their respective $\Gamma$-orbits.

If $q_F : B \to F$ denotes the $F$-quadratic form associated with our two fixed embeddings $\alpha_i : \mathcal{O}_i \xhookrightarrow{} \mathcal{R}$ for $i \in \{ 1,2 \}$, we let $q_F'$ denote the $F$-quadratic form associated with the embeddings $\overline{\alpha_1}$ and $\alpha_2$, where $\overline{ \alpha_1 }$ denotes the same embedding $\alpha_1$ precomposed with complex conjugation on $\mathcal{O}_1$. By Lemma 4.3 in \cite{daas1}, the numbers $q_F(b), q_F'(b) \in F^+ \cup \{ 0 \}$ are Galois conjugates for all $b \in B$. Extending this notation, if $\nu \in F$, we let $\nu' = \sigma( \nu )$ denote the Galois conjugate.

One can show that $q_F$ admits the explicit formula
\[
q_F(b) = \text{Nm}(b) \frac{ ( \tau_1 - b \tau_2 )( \tau_1' - b \tau_2' ) }{ ( \tau_1 - \tau_1' ) ( b \tau_2 - b \tau_2' ) } \quad \text{for all $b \in B$};
\]
see Theorem 4.4.9 in \cite{thesis}. It follows that 
\[
\frac{q_F(b)}{q_F'(b)} = \frac{ ( \tau_1 - b \tau_2 )( \tau_1' - b \tau_2' ) }{ ( \tau_1' - b \tau_2 )( \tau_1' - b \tau_2 ) } \quad \text{for all $b \in B^{\times}$},
\]
which should remind the reader of Equation \ref{maineq}. The following is true independent of any assumptions on the genus of the various Shimura curves under consideration, and it reflects the independence of the value from Equation \ref{maineq} from any choices regarding the representing CM points $\tau_1, \tau_1', \tau_2, \tau_2' \in \mathcal{H}_p$. 

Recall that Theorem \ref{bestcount} associates to the pair $(\tau_1, \tau_2)$ an ideal class $\widetilde{C}_0 \in \text{Pic}(L)$. This motivates defining
\[
\Theta_{\text{Cap}}( \tau_1, \tau_2 ) := \prod_{(c_1, c_2) \in \text{Cap}} \frac{\Theta(c_2 \cdot \tau_2, c_2 \cdot \tau_2'; c_1 \cdot \tau_1)}{\Theta( c_2 \cdot \tau_2, c_2 \cdot \tau_2'; c_1 \cdot \tau_1')}.
\] 
We stress that $\# \text{Cap} \in \{ 1,2 \}$, so this operation is far less drastic than taking the full norm, and in many cases, it even makes no difference compared to studying the individual values of the $\Theta$-function.

\begin{proposition}\label{thetarw}
It holds that
\[
\frac{2}{w_1w_2} \log_p( \Theta_{\emph{Cap}}( \tau_1, \tau_2 ) ) = \lim_{n \to \infty} \sum_{ \substack{ \nu \in (\mathcal{D}_F^{-1}\mathfrak{q})^+ \\ \emph{tr}(\nu) = p^{2n} } } \log_p \left( \frac{\nu}{\nu'} \right) \cdot \rho_{\widetilde{C}_0}( \nu \mathcal{D}_F \mathfrak{q}^{-1} ).
\]
\end{proposition}
\begin{proof}
This generalises Theorem 4.5 in \cite{daas1}, which is an average version of this result under the additional assumption that there is a unique maximal order in $B$ up to conjugation. We refine the argument here to the $p$-inverted setting. From the discussion above and taking the $p$-adic logarithm, we obtain
\[
\log_p( \Theta( \tau_1, \tau_2 )_{\text{Cap}} ) = \sum_{ (c_1,c_2) \in \text{Cap} }  \sum_{\gamma \in \Gamma} \log_p \left( \frac{q_F[c_1,c_2](\gamma)}{q_F'[c_1,c_2](\gamma)} \right).
\]
We switch the order of summation; instead of recording for each triple $( \gamma, c_1, c_2 ) \in \Gamma \times \text{Cap}$ the value of $q_F[c_1,c_2](\gamma)$, we instead sum over all the possible values $\nu \in F^+$ and record how often each value is attained;
\[
\log_p( \Theta(D_1, D_2) ) = \sum_{\nu \in F^+} \log_p( \nu / \nu' ) \cdot \# \{ ( \gamma, c_1, c_2 ) \in \Gamma \times \text{Cap} \mid q_F[c_1,c_2]( \gamma ) = \nu \}.
\] 
From Corollary \ref{idealnorm2}, it follows that the $\mathcal{O}_F[1/p]$-ideal $q_F[c_1,c_2]( \gamma ) \mathcal{D}_F \mathfrak{q}^{-1}$ is integral for all triples $( \gamma, c_1, c_2 )$ as above. Thus, if we let $X \colonequals \mathcal{D}_F^{-1} \mathfrak{q} \mathcal{O}_F[1/p] \subset F$, only those $\nu \in X^+$ can contribute to the sum above. In addition, as $\gamma \in \Gamma$ means that $\text{tr}( q_F[c_1,c_2]( \gamma ) ) = \text{Nm}( \gamma ) = 1$, we also get no contribution unless $\text{tr}( \nu ) = 1$. Conversely, if $b \in \mathcal{R}$ is such that $q_F[c_1,c_2]( b )$ has trace 1, then $\text{Nm}( b ) = 1$ and therefore $b \in \Gamma$. It follows that
\[
\log_p( \Theta(D_1, D_2) ) = \sum_{ \substack{ \nu \in X^+ \\ \text{tr}( \nu ) = 1 } } \log_p( \nu / \nu' ) \cdot \# \{ ( b, c_1, c_2 ) \in \mathcal{R} \times \text{Cap} \mid q_F[c_1,c_2]( b ) = \nu \}.
\]
For every integer $n \geq 0$, let $X^+(n) \colonequals \{ \nu \in X^+ \mid p^{2n} \nu \in (\mathcal{D}_F^{-1} \mathfrak{q})^+ \}$. There are inclusions $X^+(n) \xhookrightarrow{} X^+(n+1)$ for all $n \geq 0$, and we see that
\[
X^+ = \bigcup_{n=1}^{\infty} X^+(n).
\]
Multiplication by $p^{2n}$ induces a bijection between $X^+(n)$ and $\{ \nu \in ( \mathcal{D}_F^{-1} \mathfrak{q} )^+ \mid \text{tr}( \nu ) = p^{2n} \}$. In addition, the number of triples $(b, c_1, c_2) \in \mathcal{R} \times \text{Cap}$ such that $q_F[c_1,c_2]( b ) = \nu$ is the same as the number of triples such that $q_F[c_1,c_2]( b ) = p^{2n} \nu$ using the bijective substitution $b \mapsto p^n b$ in $\mathcal{R}$. It follows that
\[
\log_p( \Theta(D_1, D_2) ) = \lim_{n \to \infty} \sum_{ \substack{ \nu \in ( \mathcal{D}_F^{-1} \mathfrak{q} )^+ \\ \text{tr}( \nu ) = p^{2n} } } \log_p( \nu / \nu' ) \cdot \# \{ ( b, c_1, c_2 ) \in \mathcal{R} \times \text{Cap} \mid q_F[c_1,c_2]( b ) = \nu \}.
\]
As now $\nu$ is $p$-integral, it follows from Theorem \ref{bestcount} that this count equals $\frac{w_1w_2}{2} \rho_{\widetilde{C}_0}( \nu \mathcal{D}_F \mathfrak{q}^{-1} )$.
\end{proof}

There is an entirely analogous formula for the quantity $\Theta_{p, \text{cap}}( \tau_1, \tau_2 )$, defined in the obvious way, where instead the sum is taken over all $\nu \in ( \mathcal{D}_F^{-1} \mathfrak{q} )^+$ of trace $p^{2n+1}$, as $n$ tends to infinity. Taking the product over the full class groups instead, we obtain from Corollary \ref{finalcount} that
\begin{equation} \label{rewriteeq}
\frac{2}{w_1w_2} \log_p \left( \Theta(D_1, D_2) \right) = \lim_{n \to \infty} \sum_{ \substack{ \nu \in (\mathcal{D}_F^{-1}\mathfrak{q})^+ \\ \text{tr}(\nu) = p^{2n} } } \log_p \left( \frac{\nu}{\nu'} \right) \cdot \rho( \nu \mathcal{D}_F \mathfrak{q}^{-1}  ),
\end{equation}
with again an obvious analogue for $\Theta_p( D_1, D_2 )$. To prove our main Theorem \ref{mainthm}, we will further analyse the cusp form constructed in \cite{daas1}. Let us denote this by
\[
f_{\text{GD}} \colonequals e^{\text{ord}}\left( \frac{d}{d\epsilon} E^{(p)}_{1,\chi}(z,z; \epsilon) \right) \in \mathcal{S}_2( \Gamma_0(N) ).
\]
For any integral ideal $J$ of $F$, we let $\widetilde{J}$ denote its $p$-deprivation, which is obtained by removing all factors of $\mathfrak{p}_1$ and $\mathfrak{p}_2$ from the factorisation of $J$. Recall that $\mathfrak{q}$ denotes the reflex ideal associated with the embeddings $\alpha_i \colon \mathcal{O}_i \xhookrightarrow{} R_q$. For any $\nu \in (\mathcal{D}_F^{-1}\mathfrak{q})^+$, let $J_{\nu}$ denote the ideal $\nu \mathcal{D}_F \mathfrak{q}^{-1}$. 

Recall the genus character $\chi : \text{Pic}(F)^+ \to \{ \pm 1 \}$. We will refer to those prime powers $\mathfrak{l}^n \| J$ with $\chi( \mathfrak{l}^n ) = -1$ as the \emph{special primes} of an ideal $J \subset \mathcal{O}_F$. Having at least one special prime is a necessary and sufficient condition to \emph{not} be the norm of an $\mathcal{O}_L$-ideal.

If $J$ is coprime to $p$, we let $\mathcal{F}( J ) \in \mathbb{N}$ be the unique positive integer such that
\[
\log_p(\mathcal{F}(J)) \colonequals \begin{cases} 2k \cdot \rho(J / \mathfrak{l}^{2k-1}) \cdot \log_p(\text{Nm}(\mathfrak{l})) &\text{if $\mathfrak{l}^{2k-1} \| J$ is the unique special prime of $J$;} \\ 0 &\text{otherwise.} \end{cases}
\]
\begin{remark}\label{Fdef}
The function $\mathfrak{ F } : \mathbb{N} \to \mathbb{N}$ from Theorem \ref{mainthm} is related to the function $\mathcal{F}( - )$ defined above as follows. Let $J \subset \mathcal{O}_F$ be a \emph{primitive} ideal, meaning that it is not divisible by any rational prime. Then $\mathcal{F}( J ) = \mathfrak{F}( \text{Nm}( J ) )^2$; see Proposition 4.15 in \cite{daas1}. It follows that to compute the values $\mathfrak{F}( (D - x^2) / 4N )$ that occur in Theorem \ref{mainthm}, it suffices to compute the values $\mathcal{F}( \widetilde{ J }_{\nu} )$, where $\nu = ( x + \sqrt{D} ) / 2 \sqrt{D} \in ( \mathcal{D}_F^{-1} \mathfrak{q} )^+$. Indeed, as $\text{tr}( \nu ) = 1$, the ideal $J_{\nu}$ is certainly primitive and the formula $\mathcal{F}( J_{\nu} ) = \mathfrak{F}( \text{Nm}( J_{\nu} ) )^2$ applies.
\end{remark}

In \cite{daas1}, the following theorem was proved during the discussion following Theorem 4.16.
\begin{theorem}\label{fourier}
Let $m \geq 1$ be a positive integer. Then
\[
a_m( f_{\emph{GD}} ) = \lim_{n \to \infty} \sum_{\substack{\nu \in (\mathcal{D}_F^{-1}\mathfrak{q})^+ \\ \emph{tr}(\nu) = m p^{2n}}} (-1)^{\emph{ord}_{\mathfrak{p}}(\nu)} \big(\log_p( \mathcal{F}( \widetilde{J}_{\nu})) - \rho( \widetilde{J}_{\nu} )\log_p( \nu / \nu' ) \big).
\]
\end{theorem}

As replacing $\nu$ by $p \nu$ leaves $\widetilde{J}_{\nu}$ and $\nu / \nu'$ invariant, but changes the sign $(-1)^{\text{ord}_{\mathfrak{p}}( \nu )}$, it is immediate from Theorem \ref{fourier} that $f_{\text{GD}}$ is a $U_p$-eigenform with eigenvalue $-1$, and this was used in \cite{daas1} to complete the proof in the case that $N = 22$. We will show, however, that in addition, the form $f_{\text{GD}}$ is also a $U_q$-eigenform with eigenvalue $-1$, where $N = pq$. Interestingly, the nature of this symmetry is rather different.

Key are the following observations, which are also heavily exploited in \cite{daas1}. As $\chi$ is totally odd,
\[
\chi( J_{\nu} ) = \chi( \nu ) \chi( \mathcal{D}_F ) \chi( \mathfrak{q}^{-1} ) = 1 \cdot (-1) \cdot (-1) = 1,
\]
and therefore the number of special primes of $J_{\nu}$ is even. 

Therefore, the only way for $\widetilde{J}_{\nu}$ to have exactly one special prime is if exactly one of $\mathfrak{p}$ and $\mathfrak{p}'$ is special for $J_{\nu}$. In other words, if $\text{ord}_{\mathfrak{p}}( \nu )$ and $\text{ord}_{\mathfrak{p}}( \nu )$ have the same parity, then $\log_p( \mathcal{F}( \widetilde{J}_{\nu})) = 0$. 

Furthermore, $\rho( \widetilde{J}_{\nu} )$ vanishes as soon as $\widetilde{J}_{\nu}$ has a special prime. It follows that if $\text{ord}_{\mathfrak{p}}( \nu )$ and $\text{ord}_{\mathfrak{p}'}( \nu )$ have \emph{opposite} parities, so exactly one of $\mathfrak{p}$ and $\mathfrak{p}'$ is special for $J_{\nu}$, then we have $\rho( \widetilde{J}_{\nu} ) = 0$.

\begin{proposition}\label{Uq-1}
The cusp form $f_{\text{GD}} \in \mathcal{S}_2( N )$ is a $U_q$-eigenform with eigenvalue $-1$.
\end{proposition}
\begin{proof}
We will show that $f_{\text{GD}}$ was already a $U_q$-eigenform before ordinary projection. The $m$th Fourier coefficient of this object is equal to 
\[
b_m \colonequals \sum_{\substack{\nu \in (\mathcal{D}_F^{-1}\mathfrak{q})^+ \\ \text{tr}(\nu) = m }} (-1)^{\text{ord}_{\mathfrak{p}}(\nu)} \big(\log_p( \mathcal{F}( \widetilde{J}_{\nu})) - \rho( \widetilde{J}_{\nu} )\log_p( \nu / \nu' ) \big).
\]
We will work towards showing that $b_{qm} = -b_m$ for any positive integer $m$. In fact, we define
\[
f_m \colonequals \sum_{\substack{\nu \in (\mathcal{D}_F^{-1}\mathfrak{q})^+ \\ \text{tr}(\nu) = m }} (-1)^{\text{ord}_{\mathfrak{p}}(\nu)} \log_p( \mathcal{F}( \widetilde{J}_{\nu})) \quad \text{and} \quad t_m \colonequals \sum_{\substack{\nu \in (\mathcal{D}_F^{-1}\mathfrak{q})^+ \\ \text{tr}(\nu) = m }} (-1)^{\text{ord}_{\mathfrak{p}}(\nu)} \rho( \widetilde{J}_{\nu} )\log_p( \nu / \nu' ).
\]
As $b_m = f_m - t_m$ for all $m \geq 1$, it now suffices to show that $f_{qm} = - f_m$ and $t_{qm} = - t_m$. 

Let $\nu \in ( \mathcal{D}_F^{-1} \mathfrak{q} )^+$ be such that $\text{tr}( \nu ) = qm$. Since $q$ divides $\nu + \nu'$ and $\mathfrak{q}$ divides $\nu$, it follows that $\mathfrak{q}$ divides $\nu'$ as well and therefore $\mathfrak{q}'$ divides $\nu$. So, $q \mid \nu$ and we may write $\nu = q \mu'$ for some $\mu \in \mathcal{D}_F^{-1,+}$ with $\text{tr}( \mu ) = m$. 

Let us first focus on $t_{qm}$. If $\text{ord}_{\mathfrak{p}}( \nu )$ and $\text{ord}_{\mathfrak{p}'}( \nu )$ have opposite parities, then $\rho( \widetilde{J}_{\nu} ) = 0$, so these elements will not contribute to $t_{qm}$. We now study those $\nu$ for which $\text{ord}_{\mathfrak{p}}( \nu )$ and $\text{ord}_{\mathfrak{p}'}( \nu )$ have the same parity. 

For the value $\rho( \widetilde{J}_{\nu} ) = 0$ to be non-zero, $\mathfrak{q}'$ cannot be a special prime of this ideal. As $\widetilde{J}_{\nu} = \widetilde{\mu}' \mathcal{D}_F \mathfrak{q}'$, it follows that $\text{ord}_{\mathfrak{q}'}( \mu' )$ must be odd, and thus in particular at least 1. Therefore $\mu' \in ( \mathcal{D}_F^{-1} \mathfrak{q}' )^+$, and so $\mu \in ( \mathcal{D}_F^{-1} \mathfrak{q} )^+$. Next, we use that applying the Galois action of $F / \mathbb{Q}$ to an ideal does not change its value 
\renewcommand{\baselinestretch}{0.96}\normalsize
under the function $\rho$, and the conjugate of $\widetilde{J}_{\nu}$ is $( \mu ) \mathcal{D}_F \mathfrak{q} = \mathfrak{q}^2 \cdot \widetilde{J}_{\mu}$. It follows that $\rho( \widetilde{J}_{\nu} ) = \rho( \mathfrak{q}^2 \widetilde{J}_{\mu} ) = \rho( \widetilde{J}_{\mu} )$, where we used that the ideal $\widetilde{J}_{\mu}$ is integral. As $\text{ord}_{\mathfrak{p}}( \nu ) = \text{ord}_{\mathfrak{p}}( \mu' ) \equiv \text{ord}_{\mathfrak{p}'}( \mu' ) = \text{ord}_{\mathfrak{p}}( \mu ) \mod 2$, we find
\[
t_{qm} = \sum_{\substack{\mu \in (\mathcal{D}_F^{-1}\mathfrak{q})^+ \\ \text{tr}(\nu) = m }} (-1)^{\text{ord}_{\mathfrak{p}}(\mu)} \rho( \widetilde{J}_{\mu} )\log_p( \mu' / \mu ) = - t_m,
\]
as $\log_p( \mu' / \mu ) = - \log_p( \mu / \mu' )$, proving the first assertion. 

We now turn to $f_{qm}$, and observe that if $\text{ord}_{\mathfrak{p}}( \nu )$ and $\text{ord}_{\mathfrak{p}'}( \nu )$ have the same parity, then $\log_p( \mathcal{F}( \widetilde{J}_{\nu})) = 0$, so we reduce to studying those $\nu$ for which $\text{ord}_{\mathfrak{p}}( \nu )$ and $\text{ord}_{\mathfrak{p}'}( \nu )$ have opposite parities. The quantity $\log_p( \mathcal{F}( \widetilde{J}_{\nu}))$ is only non-zero if $\widetilde{J}_{\nu}$ has a unique special prime. The prime $\mathfrak{q}$ is special if and only if $\text{ord}_{\mathfrak{q}}( \nu )$ is even, and $\mathfrak{q}'$ is special if and only if $\text{ord}_{\mathfrak{q}'}( \nu )$ is odd. This suggests splitting up the quantity $f_{qm}$ even further; we define
\[
A_0 \colonequals \sum_{\substack{\nu \in (\mathcal{D}_F^{-1}\mathfrak{q})^+ \\ \text{tr}(\nu) = qm \\ \text{ord}_{\mathfrak{q}'}( \nu ) \text{ even } }} (-1)^{\text{ord}_{\mathfrak{p}}(\nu)} \log_p( \mathcal{F}( \widetilde{J}_{\nu})) \quad \text{and} \quad A_1 \colonequals \sum_{\substack{\nu \in (\mathcal{D}_F^{-1}\mathfrak{q})^+ \\ \text{tr}(\nu) = qm \\ \text{ord}_{\mathfrak{q}'}( \nu ) \text{ odd } }} (-1)^{\text{ord}_{\mathfrak{p}}(\nu)} \log_p( \mathcal{F}( \widetilde{J}_{\nu})),
\]
so that $A_0 + A_1 = f_{qm}$. Let us first compute $A_0$, in which case $\text{ord}_{\mathfrak{q}'}( \mu' )$ is odd. In particular it is at least $1$ and we have $\mu' \in ( \mathcal{D}_F^{-1} \mathfrak{q}' )^+$, and so $\mu \in ( \mathcal{D}_F^{-1} \mathfrak{q} )^+$. Assuming $\widetilde{J}_{\nu}$ has a unique special prime, which is now different from $\mathfrak{q}'$, the Galois conjugate ideal $\mathfrak{q}^2 \cdot \widetilde{J}_{\mu}$ also has a unique (Galois conjugate) special prime, now different from $\mathfrak{q}$, and thus so does the integral ideal $\widetilde{J}_{\mu}$. It follows from the definition of $\mathcal{F}$ that now $\mathcal{F}( \widetilde{J}_{\nu} ) = \mathcal{F}( \mathfrak{q}^2 \cdot \widetilde{J}_{\mu} ) = \mathcal{F}(  \widetilde{J}_{\mu} )$. As now $\text{ord}_{\mathfrak{p}}( \nu ) = \text{ord}_{\mathfrak{p}}( \mu' ) \not\equiv \text{ord}_{\mathfrak{p}'}( \mu' ) = \text{ord}_{\mathfrak{p}}( \mu ) \mod 2$, we find
\[
A_0 = - \sum_{\substack{\mu \in (\mathcal{D}_F^{-1}\mathfrak{q})^+ \\ \text{tr}(\mu) = m \\ \text{ord}_{\mathfrak{q}}( \mu ) \text{ odd } }} (-1)^{\text{ord}_{\mathfrak{p}}(\mu)} \log_p( \mathcal{F}( \widetilde{J}_{\mu})).
\]
Finally, we turn to $A_1$. By construction, for each term in the sum, the prime $\mathfrak{q}'$ is special for $\widetilde{J}_{\nu}$, and thus $\mathcal{F}( \widetilde{J}_{\nu})$ is either $1$ or a positive power of $q$. If $\mathfrak{q}$ were also special, then $\log_p( \mathcal{F}( \widetilde{J}_{\nu} ) ) = 0$, so we may assume that $\text{ord}_{\mathfrak{q}}( \nu ) = \text{ord}_{\mathfrak{q}'}( \nu' )$ is also odd. For convenience, we introduce $\text{ord}_q \colonequals ( \text{ord}_{\mathfrak{q}}, \text{ord}_{\mathfrak{q}'} )$. As $q \mid \nu$, also $\nu' \in ( \mathcal{D}_F^{-1} \mathfrak{q} )^+$ and therefore the set over which the sum defining $A_1$ is taken is stable under Galois conjugation. As $\text{ord}_{\mathfrak{p}}( \nu )$ and $\text{ord}_{\mathfrak{p}}(\nu') = \text{ord}_{\mathfrak{p}'}( \nu )$ are assumed to have opposite parities, we find
\[
A_1 = \frac{1}{2} \sum_{\substack{\nu \in (\mathcal{D}_F^{-1}\mathfrak{q})^+ \\ \text{tr}(\nu) = qm \\ \text{ord}_{q}( \nu ) \equiv (1,1) \bmod 2 } } (-1)^{\text{ord}_{\mathfrak{p}}(\nu)} \big( \log_p( \mathcal{F}( \widetilde{J}_{\nu})) - \log_p( \mathcal{F}( \widetilde{J}_{\nu'})) \big).
\] 
The prime $\mathfrak{q}$ is special for the Galois conjugate $\mathfrak{q}^2 \cdot \widetilde{J}_{\mu}$ of $\widetilde{J}_{\nu}$, and we find that $\mathcal{F}( \widetilde{J}_{\nu} ) = \mathcal{F}( \mathfrak{q}^2 \cdot \widetilde{J}_{\mu} )$. Let us momentarily denote $\text{ord}_{\mathfrak{q}}( \mu ) = 2a$ and $\text{ord}_{\mathfrak{q}'}( \mu ) = 2b$ for certain integers $a,b \geq 0$. If we further momentarily abbreviate $\rho = \rho( \widetilde{J}_{\mu} / \mathfrak{q} )$, then by definition, $\log_p( \mathcal{F}( \mathfrak{q}^2 \cdot \widetilde{J}_{\mu} ) ) = (a+1)\rho \log_p( q )$, whereas $\log_p( \mathcal{F}( \widetilde{J}_{\mu} ) ) = a \rho \log_p( q )$. Similarly, $\log_p( \mathcal{F}( \mathfrak{q}^2 \cdot \widetilde{J}_{\mu'} ) ) = (b+1)\rho \log_p( q )$, whereas $\log_p( \mathcal{F}( \widetilde{J}_{\mu'} ) ) = b \rho \log_p( q )$. Therefore,
\[
\log_p( \mathcal{F}( \widetilde{J}_{\nu})) - \log_p( \mathcal{F}( \widetilde{J}_{\nu'})) = (a-b) \rho \log_p( q ) = \log_p( \mathcal{F}( \widetilde{J}_{\mu})) - \log_p( \mathcal{F}( \widetilde{J}_{\mu'})).
\]
We have therefore shown that
\[
A_1 = - \frac{1}{2} \! \sum_{\substack{\mu \in (\mathcal{D}_F^{-1})^+ \\ \text{tr}(\nu) = m \\ \text{ord}_{q}( \mu ) \equiv (0,0) \bmod 2 } } (-1)^{\text{ord}_{\mathfrak{p}}(\mu)} \big( \log_p( \mathcal{F}( \widetilde{J}_{\mu})) - \log_p( \mathcal{F}( \widetilde{J}_{\mu'})) \big) =  - \! \sum_{\substack{\mu \in (\mathcal{D}_F^{-1} \mathfrak{q})^+ \\ \text{tr}(\nu) = m \\ \text{ord}_{\mathfrak{q}}( \mu ) \text{ even } }} (-1)^{\text{ord}_{\mathfrak{p}}(\mu)} \log_p( \mathcal{F}( \widetilde{J}_{\mu})),
\]
where we restricted our sum to those $\mu \in (\mathcal{D}_F^{-1} \mathfrak{q})^+$ instead of all $\mu \in \mathcal{D}_F^{-1,+}$ as the value of $\mathcal{F}$ is zero on any non-integral ideal. We may then conclude that
\[
f_{qm} = A_0 + A_1 = - \sum_{\substack{\mu \in (\mathcal{D}_F^{-1} \mathfrak{q})^+ \\ \text{tr}(\nu) = m }} (-1)^{\text{ord}_{\mathfrak{p}}(\mu)} \log_p( \mathcal{F}( \widetilde{J}_{\mu})) = -f_m. \vspace{-6mm}
\]
\end{proof}
\renewcommand{\baselinestretch}{1.00}\normalsize
\begin{theorem}
Suppose that the genus of $X_N / \emph{w}_N$ is zero. Then
\[
\left( \frac{\Theta( D_1, D_2 )}{\Theta_p( D_1, D_2 )} \right)^{\frac{\pm 2}{w_1w_2}} = \pm \prod_{\substack{x^2 < D \\ x^2 \equiv D \emph{ mod } 4N}} \mathfrak{F}\left( \frac{D-x^2}{4N} \right)^{\delta(x)}.
\]
\end{theorem}
\begin{proof}
From Equation \ref{rewriteeq} and Theorem \ref{fourier} one may deduce, see Subsection 4.4 in \cite{daas1}, that the first Fourier coefficient of $f_{\text{GD}}$ is given by the difference between the $p$-adic logarithms of both sides of the equation we are trying to prove. In addition, we have remarked that $f_{\text{GD}}$ is in fact a $U_p$-eigenform with eigenvalue $-1$, and by Proposition \ref{Uq-1}, it is also a $U_q$-eigenform with eigenvalue $-1$. Under our assumption, it now follows from Propositions \ref{equiv1} and \ref{no-1} that $f_{\text{GD}} = 0$. In particular, its first Fourier coefficient vanishes, and we obtain an equality between the $p$-adic logarithms of both sides of the equation from the theorem. As one may count factors of $p$ on both sides separately, as is done in Proposition 4.8 in \cite{daas1}, directly from the formula from Proposition \ref{thetarw}, we have now proved the equality up to sign, and so we are done.
\end{proof}

\subsection*{Acknowledgements}

The author would like to thank H\aa vard Damm-Johnsen, Mateo Crabit Nicolau, Oana Padurariu and Jan Vonk for helpful discussions and comments on an earlier version of this manuscript. 

The author is also indebted to the Max Planck Institute for Mathematics in Bonn, where some parts of this paper were finalised during a stay in the summer of 2026.

\bibliographystyle{alpha}
\bibliography{biblio}

\end{document}